\documentclass[12pt,a4paper]{amsart}
\usepackage[margin = 1in]{geometry}
\usepackage{amssymb}
\usepackage{amsfonts}
\usepackage{amsthm}
\usepackage{amsmath}
\usepackage{amscd}
\usepackage{esint}
\usepackage{verbatim}
\usepackage{graphicx}
\usepackage{caption}
\allowdisplaybreaks

\newcommand\norm[1]{\left\lVert#1\right\rVert}
\newtheorem{thm}{Theorem}[section]
\newtheorem{lemma}[thm]{Lemma}
\newtheorem{prop}[thm]{Proposition}

\theoremstyle{definition}
\newtheorem{definition}[thm]{Definition}

\theoremstyle{remark}

\numberwithin{equation}{section}

\title{Global regularity and fast small scale formation for Euler patch equation in a smooth domain}
\author{Alexander Kiselev${}^{*}$ and Chao Li${}^{**}$}
\thanks{\footnotesize${}^{*}$ Department of Mathematics, Duke University, Durham NC 27708, USA; email: kiselev@math.duke.edu.}
\thanks{\footnotesize${}^{**}$ Department of Mathematics, Rice University; Houston TX 77005, USA; email: chao.li@rice.edu.}

\begin{document}

\maketitle
\pagestyle{plain}

	\begin{abstract}
		It is well known that the Euler vortex patch in $\mathbb{R}^{2}$ will remain regular if it is regular enough initially.
In bounded domains, the regularity theory for patch solutions is
less complete. In this paper, we study Euler vortex patches in a general smooth bounded domain. We prove global in time regularity
by providing an upper bound on the growth of curvature of the patch boundary.
For a special symmetric scenario, we construct an example of double exponential curvature growth, showing that our upper bound is
qualitatively sharp.
	\end{abstract}

	\section{Introduction}

	The incompressible Euler equation in a compact domain $D \subset \mathbb{R}^d$ with natural no-penetration boundary
conditions is given by
\[ \partial_t u + (u \cdot \nabla)u = \nabla p, \,\,\, \nabla \cdot u =0, \,\,\, \left. u \cdot n \right|_{\partial D} =0. \]
	We will only deal with the case $d=2,$ and it will be conveninient for us to work with the equation in vorticity
form. Thus, setting $\omega = {\rm curl} u,$ we have
	\begin{equation}\label{euler}
	\partial_t \omega + (u\cdot\nabla)\omega = 0,
	\end{equation}
	
	\begin{equation}\label{BSL}
	u = \nabla^{\perp}(-\Delta_{D})^{-1}\omega.
	\end{equation}
	
	Here $\nabla^{\perp}$ and $x^{\perp}$ denote $(\partial_{2},-\partial_{1})$ and $(x_{2},-x_{1})$ respectively, and $\Delta_{D}$ is the
 Dirichlet Laplacian (see e.g \cite{Majda, Marchioro}). Equation (\ref{BSL}) is called the Biot-Savart law. \\
	
	Global regularity for the $2D$ Euler equation with smooth initial data has been known since 1930s \cite{Wolibner}.
In this paper, we are interested in a class of rough solutions called vortex patches. These solutions need to be understood
in an appropriate weak sense.
	
	Assuming for a moment that $u$ is sufficiently regular, define particle trajectories associated to the vector field $u$ by
	
	\begin{equation}\label{particletrajectory}
	\dfrac{d}{dt}\Phi_{t}(x) = u(\Phi_{t}(x),t), \quad \Phi_{0}(x) = x.
	\end{equation}
	
	Since the active scalar $\omega$ is transported by the velocity $u$ from (\ref{euler}), it is straightforward to check that
	
	\begin{equation}\label{advection}
	\omega(x,t) = \omega_{0}(\Phi^{-1}_{t}(x)).
	\end{equation}
	
	The solution satisfying (\ref{BSL}), (\ref{particletrajectory}) and (\ref{advection}) is called a solution to the
Euler equation in Yudovich sense (see \cite{Yudovich,Majda}).
	
	An Euler vortex patch is a solution to the Euler equation in Yudovich sense of the form
	
	\begin{equation}
	\omega(x,t) = \sum_{k = 1}^{N}\theta_{k}\chi_{\Omega_{k}(t)}(x).
	\end{equation}
	
	Here $\theta_{k}$ are some constants, and $\Omega_{k}(t)$ are (evolving in time) bounded open sets in $D$ with smooth (in some sense) boundaries, whose
closures $\overline{\Omega}_{k}(t)$ are mutually disjoint.\\
	
	It is well known that Yudovich solutions to the
	2D Euler equation with initial data in $L^{\infty}\cap L^{1}$ exist and are unique (see \cite{Yudovich} or \cite{Majda,Marchioro} for a modern proof).
The main reason behind this result is the log-Lipschitz control on the velocity $u,$ which allows to define the trajectories uniquely and derive
appropriate estimates on the flow map $\Phi_t.$
In this paper, we study a stronger notion of patch regularity which refers to sufficient smoothness of the patch boundaries $\partial\Omega_{k}$,
as well as to the lack of both self-intersections of each patch boundary and touching of different patches.

	To be precise, we have the following definitions.
	
	\begin{definition}
		Let $\Omega \subseteq D$ be an open set whose boundary $\partial\Omega$ is a simple closed $C^{1}$ curve with arc-length $|\partial\Omega|$. A constant speed parametrization of $\partial\Omega$ is any counter-clockwise parametrization $z : \mathbb{T} \rightarrow \mathbb{R}^{2}$ of $\partial\Omega$ with $|z^{\prime}| = \frac{1}{2\pi}|\partial\Omega|$ on the circle $\mathbb{T} := [-\pi,\pi]$ (with $\pm\pi$ identified), and we define $\norm{\Omega}_{C^{m,\gamma}} :=  \norm{z}_{C^{m,\gamma}}$.
	\end{definition}
	
	\begin{definition}\label{generalpatchdef}
		Let $\theta_{1},...,\theta_{N} \in \mathbb{R}\setminus\{0\}$, and for each $t \in [0,T)$, let $\Omega_{1}(t),...,\Omega_{N}(t) \subseteq D$ be open sets with pairwise disjoint closures whose boundaries $\partial\Omega_{k}(t)$ are simple closed curves. Let
		\[
		\omega(x,t) := \sum\limits_{k=1}^{N}\theta_{k}\chi_{\Omega_{k}(t)}(x).
		\]
		Suppose $\omega$ also satisfies
		\[
		\omega(x,t) = \omega_{0}(\Phi^{-1}_{t}(x))
		\]
		where $\dfrac{d}{dt}\Phi_{t}(x) = u(\Phi_{t}(x),t)$, $\Phi_{0}(x) = x$, and $u$ is given by (\ref{BSL}). Then $\omega$ is called a patch solution to (\ref{euler}) and (\ref{BSL}) with initial data $\omega_{0}$ on the interval $[0,T)$. In addition, if we also have
		\[
		\sup\limits_{t \in [0,T^{\prime}]}\norm{\Omega_{k}(t)}_{C^{m,\gamma}} \, < \infty
		\]
		for each $k$ and $T^{\prime} \in (0,T)$, then $\omega$ is a $C^{m,\gamma}$ patch solution to (\ref{euler}) and (\ref{BSL}) on $[0,T)$.
		
	\end{definition}
	
	\textbf{Remark}. In the above definition, the domains $\Omega_{k}(t)$ are allowed to touch $\partial D$ as long as $\partial\Omega_{k}(t)$ remain $C^{m,\gamma}$.\\

	Possible singularity formation for two dimensional Euler vortex patches has been conjectured based on the numerical simulations
in \cite{Buttke} (see \cite{AMajda} for a discussion). In 1993, Chemin \cite{Chemin} proved that the boundary of a two dimensional Euler patch will remain regular for all time
if it is regular enough ($C^{1,\gamma}$) initially (see also the work by Bertozzi and Constantin in \cite{Constantin} for a different proof).
For vortex patches in domains with boundaries, Depauw \cite{Depauw} has proved global regularity of a single $C^{1,\gamma}$ patch in the half plane when the patch does
not touch the boundary initially. If the initial patch touches the boundary, then \cite{Depauw} proved that the regularity $C^{1,\gamma}$ will be retained for a finite time.
Dutrifoy \cite{Dutrifoy} proved that for the initial patch touching the boundary, there is a global solution but in a strictly weaker space $C^{1,s}$ for
some $s \in (0,\gamma)$. Recently, Kiselev, Ryzhik, Yao and Zlato$\check{s}$ \cite{KRYZ} have proved global regularity for two dimensional $C^{1,\gamma}$
Euler vortex patch solutions (which may involve multiple patches) in half plane without loss of regularity. \\

	Our goal here is to explore Euler patch dynamics in a general smooth bounded domain. We derive global upper bounds on growth of curvature as well
as construct an example showing sharpness of the upper bound in some special scenarios with symmetry.\\
		
	Here are our main results.

	\begin{thm}\label{singlepatch}
		Let $D$ be a $C^{4}$ bounded domain, $\gamma \in (0,1)$, then for each $C^{1,\gamma}$ single patch initial data $\omega_{0}$, there exists
a unique global $C^{1,\gamma}$ patch solution $\omega$ to (\ref{euler}) and (\ref{BSL}) with $\omega(\cdot,0) = \omega_{0}$. The curvature of the patch
boundary grows at most double exponentially.
	\end{thm}
	
	\begin{thm}\label{generalpatch}
		Let $D$ be a $C^{4}$ bounded domain, $\gamma \in (0,1)$, then for each $C^{1,\gamma}$ patch initial data $\omega_{0}$, there exists a unique global regular $C^{1,\gamma}$ patch solution $\omega$ to (\ref{euler}) and (\ref{BSL}) with $\omega(\cdot,0) = \omega_{0}$. The curvature of boundary grows at most triple exponentially.	
	\end{thm}
	
	In a special case where the domain $D$ is a unit disk and the initial patch is odd with resepct to $x_{2}$ axis and consists of two symmetric single patches,
we have a sharp upper bound estimate on the curvature growth.
	
	\begin{thm}\label{symmetriccase}
		Let $D:= B_{1}(0)$ be a unit disk centered at the origin, and let $\gamma \in (0,1)$. Suppose that the initial data has the form
$\omega_{0}(x) = \chi_{\Omega_{1}}(x) - \chi_{\Omega_{2}}(x)$, where $\Omega_{1} \subset \{(x_{1}, x_{2}) : x_{1} \geq 0\}$ is connected
and $\Omega_{2} \subset \{(x_{1}, x_{2}) : x_{1} \leq 0\}$ is its reflection with respect to the $x_2$ axis.
Then for each $C^{1,\gamma}$ initial data $\omega_{0}$ of this form, there exists a unique global
$C^{1,\gamma}$ patch solution $\omega$ to (\ref{euler}) and (\ref{BSL}) with $\omega(\cdot,0) = \omega_{0}$. The curvature of the patch boundary grows at most double exponentially.	
	\end{thm}
	
	\begin{thm}\label{sharpexample}
In the same setting  as in Theorem \ref{symmetriccase},
there exist an $\omega_{0}$ in $C^{1,\gamma}$ such that the curvature of the boundary of the corresponding patch solution does grow at a double exponential speed.
	\end{thm}

	\textbf{Remarks}. 1. To avoid excessive technicalities,
we did not make an effort to optimize the $C^4$ regularity assumption on the domain $D$ in Theorems~\ref{singlepatch}, \ref{generalpatch}. \\
2. It is not clear whether the triple exponential upper bound of Theorem~\ref{generalpatch} is sharp.
We have no concrete scenario for it, but at the same time improving this estimate requires non-trivial new ideas.\\
	
	The rest of the paper is organized as follows. In section 2, we give the proof of Theorem \ref{singlepatch}. In section 3, we deal with the multiple patch
case, and provide the proof of Theorem \ref{generalpatch}. In section 4, we look into the special symmetric case, and prove Theorem \ref{symmetriccase}.
In the last section, we extend the example of \cite{KS} to show that the upper bound obtained in section 4 is actually sharp, thus proving Theorem \ref{sharpexample}. Throughout the paper, we denote by $C(\gamma)$, $C(D)$, $C(r)$, $C(D, \gamma)$ etc. various constants that depend only on the arguments in the bracket. We denote by $C$ universal constants. All these constants may change from line to line.

	\section{Single patch case}

	We consider a single patch $\Omega(t)\subset D$, with
	\[
	\omega(x,t) = \theta _{0} \chi _{\Omega(t)}(x).
	\]
	Without loss of generality, we set $\theta_{0} = 1$ throughout this section.\\
	
Following the ideas of \cite{Constantin}, we reformulate vortex patch evolution in terms of the evolution of a function $\varphi(x,t)$, which defines the patch via
	\begin{equation}\label{definitionOfVarphi}
	\Omega(t) = \{x : \varphi(x,t) > 0\}.
	\end{equation}
	
	If $\partial\Omega(0)$ is a simple closed $C^{1,\gamma}$ curve, there exists a function $\varphi_{0}\in C^{1,\gamma}(\overline{\Omega}(0))$,
such that $\varphi_{0} > 0$ on $\Omega(0)$, $\varphi_{0} = 0$ on $\partial\Omega(0)$ and $\inf\limits_{\partial\Omega(0)}|\nabla\varphi_{0}|>0$.
Such $\varphi_{0}$ can be obtained, for instance, by solving the Dirichlet problem
	
\[		\begin{split}
			-\Delta\varphi_{0} & = f \text{ on } \Omega(0), \\
			\varphi_{0} & = 0 \text{ on } \partial\Omega(0),
		\end{split}   \]
	with an arbitrary $0 \leq f \in C_{0}^{\infty}(\Omega(0))$ (see discussion in \cite{KRYZ}, or \cite{hardt} for a complete proof).\\
	
In what follows, we retrace some computations done in \cite{Constantin,KRYZ}; more details can be found there.
	For $x\in \Omega(t)$, we set $\varphi(x,t) = \varphi_{0}(\Phi_{t}^{-1}(x))$, with $\Phi_{t}^{-1}$ being the inverse map of $\Phi_{t}$. Then $\varphi$ solves
	\[
	\partial_{t}\varphi + (u\cdot\nabla)\varphi = 0
	\]
	on $\{ (t,x) : t > 0 \ \textrm{and}\  x\in\Omega(t) \}$.
Therefore for each $t \geq 0$, $\varphi(\cdot,t) > 0$ on $\Omega(t)$, and vanishes on $\partial\Omega(t)$ (note that $\varphi$ is not
defined on $\mathbb{R}^{2}\setminus\overline{\Omega}(t)$). Let
	\begin{equation}\label{definitionOfW}
	w = (w_{1},w_{2})=\nabla^{\perp}\varphi=(\partial_{2}\varphi,-\partial_{1}\varphi),
	\end{equation}
	and define
	
	\begin{equation}\label{AAA}
	\begin{split}
	A_{\gamma}(t) &:= \norm{w(\cdot,t)}_{\dot{C}^{\gamma}(\Omega(t))} = \sup\limits_{x,y\in\Omega(t)}\frac{|w(x,t)-w(y,t)|}{|x-y|^{\gamma}}, \\
	A_{\infty}(t) &:= \norm{w(\cdot,t)}_{L^{\infty}(\Omega(t))},\\
	A_{\inf}(t) &:= \inf\limits_{x\in\partial\Omega(t)}|w(x,t)|.
	\end{split}
	\end{equation}
	
	By our choice of $\varphi_{0}$, we have
	\[
	A_{\gamma}(0), \: A_{\infty}(0),\: A_{\inf}^{-1}(0) \: < \infty.
	\]
	
	Since $w = \nabla^{\perp}\varphi$, we know $w$ is divergence free and one can check that it solves
	\begin{equation}\label{w-evolution}
	w_{t} + (u\cdot\nabla)w = (\nabla u)w.
	\end{equation}
	
	Using \eqref{w-evolution}, one can derive the following bounds (we refer to \cite{Constantin,KRYZ} for the details).

    \begin{equation}\label{Ainftyestimate}
    A^{\prime}_{\infty}(t) \leq C(\gamma)A_{\infty}(t)\norm{\nabla u(\cdot,t)}_{L^{\infty}(\mathbb{R}^{2})},
    \end{equation}

    \begin{equation}\label{Ainfestimate}
    A^{\prime}_{\inf}(t) \geq - C(\gamma)A_{\inf}(t)\norm{\nabla u(\cdot,t)}_{L^{\infty}(\mathbb{R}^{2})},
    \end{equation}

	\begin{equation}\label{Agammaestimate}
	A^{\prime}_{\gamma}(t) \leq C(\gamma)\norm{\nabla u(\cdot,t)}_{L^{\infty}(\mathbb{R}^{2})}A_{\gamma}(t) + \norm{\nabla u(\cdot,t)w(\cdot,t)}_{\dot{C}^{\gamma}(\Omega(t))}.
	\end{equation}

	Thus it suffices to derive appropriate bounds on $\norm{\nabla u(\cdot,t)}_{L^{\infty}(\mathbb{R}^{2})}$ and $\norm{\nabla u(\cdot,t)w(\cdot,t)}_{\dot{C}^{\gamma}(\Omega(t))}.$
 This involves some estimates on the Dirichlet Green's function $G_{D}(x,y)$.\\
	
	
Recall that $G_D(x,y)$ can be written as
\begin{equation}
G_D(x,y) = \frac{1}{2\pi} \log |x-y| + h(x,y),
\end{equation}
where
\begin{equation}
   \Delta_{x}h = 0,  \quad h|_{x\in\partial D} = -\frac{1}{2\pi}\log|x-y|.
\end{equation}

	The following proposition summarizes some of the standard estimates on $G_{D}(x,y)$ we will need
(see \cite{Marchioro,Trudinger}).

    \begin{prop}\label{estimatesofgreenfunction}
   Let  $D \subset \mathbb{R}^{2}$ be a $C^4$ bounded domain. Then the Dirichlet Green's function $G_{D}(x,y)$ satisfies the following properties:

    	\begin{equation}
    		|G_{D}(x,y)| \leq C(D)(\log|x-y| + 1),
    	\end{equation}
    	\begin{equation}
    	|\nabla^{k}G_{D}(x,y)| \leq C(D)|x-y|^{-k}, \quad k = 1,2,3.
    	\end{equation}
    	Recall that $C(D)$ is a constant only depending on $D$, changing from line to line.
    \end{prop}
   	

We will need a more detailed representation of the Green's function in the case where $x$ and $y$ are close to the
boundary $\partial D.$

   	

  First we define the symmetric reflection $\widetilde{y}$ with respect to $\partial D$ for some qualified $y$.
   	
   	\begin{definition}\label{definitionofystar}
   		Suppose $y \in D$, and there exists a unique nearest point to $y$ on $\partial D$, denoted by $e(y)$. We define $\widetilde{y} = 2e(y) - y$
to be the symmetric point of $y$ with respect to $\partial D$, and define the mapping $S: y \rightarrow \widetilde{y}$.
   	\end{definition}

   	
The first half of the following proposition is standard; see e.g. Proposition 14 in \cite{CT} for more details. The second half it is not difficult to verify and has been
proved in \cite{Xu}.

   	\begin{prop}\label{tubularneighborhood}
Let $D$ be a $C^k$ bounded domain, $k \geq 2.$ Define a tubular neighborhood $T(r)$ of $\partial D$ by $T(r)= \{y \in \mathbb{R}^{2}: d(y, \partial D) \leq r\}.$
There exists $r(D)>0$ such that if $r \leq r(D)$, then $\partial T(r)$ is $C^{k-1}$, and for any $y \in T(r)$ there exists a unique nearest point $e(y) \in \partial D.$

If $y \in T(r),$ the reflection $S(y) \equiv \widetilde{y} = 2e(y)-y$ is well defined and $C^{k-1}$ regular in all $T(r).$
   	\end{prop}
 	

   Now we state the  estimate on $G_{D}(x,y)$, with $x, y$ close to $\partial D$. This representation is a minor variation of the result derived by Xu \cite{Xu};
we will provide a sketch of the argument in the appendix.	
   	
   \begin{prop}\label{estimatesfromxiao}
   		Suppose $D \subset \mathbb{R}^{2}$ is a $C^{4}$ bounded domain. There exists $r = r(D) > 0$ such that for any $x$, $y \in T(r)$, we have
   	    \begin{equation}
   	    G_{D}(x,y) = \frac{1}{2\pi}(\log|x - y| - \log|x -\widetilde{y}|) + B(x,y).
   	    \end{equation}
   	
   	    For any $\omega \in L^{\infty}(D)$, and $0 < \alpha < 1$, $B(x,y)$ satisfies
   	    \[
   	    \displaystyle\int_{T(r)\cap D}B(x,y)\omega(y)dy \in C^{2,\alpha}(T(r)).
   	    \]
   	    More precisely, we have
   	    \begin{equation}\label{c2alphaestimateofB}
   	    \norm{\displaystyle\int_{T(r)\cap D}B(x,y)\omega(y)dy}_{C^{2,\alpha}(T(r)\cap D)} \leq C(D)\norm{\omega}_{L^{\infty}}.
   	    \end{equation}
   \end{prop}

   It is not hard to observe that, for each $\epsilon > 0$, there exists $r(\epsilon) > 0$, such that for any $p \in \partial D$, we have
   \begin{equation}\label{UnitTangentVectorDifference}
   	|t_{1} - t_{2}| \leq \epsilon,
   \end{equation}
   where $t_{1}$, $t_{2}$ are two arbitrary unit tangent vectors to $\partial D \cap B_{r}(p)$.

   Indeed, since $\partial D$ is $C^{4}$, we denote by $r_D$ the inverse of the maximal curvature of $\partial D$. Then we can choose $r(\epsilon)$ to be any positive number less than $\dfrac{r_{D}\cdot\epsilon}{2}$.

   For convenience, throughout this section, we set $0 < r \leq r(\frac{1}{100})$ (we pick $\epsilon$ in (\ref{UnitTangentVectorDifference}) to be $\frac{1}{100}$) to be small enough such that both Proposition \ref{tubularneighborhood} and \ref{estimatesfromxiao} apply.


   With the above propositions, we begin by estimating $\norm{\nabla u(\cdot,t)}_{L^{\infty}(\mathbb{R}^{2})}$.

   \begin{prop}\label{gradientofvelocity}
   	Assume $u$ is given by the Biot-Savart law formula (\ref{BSL}) and $A_{\gamma}(t)$, $A_{\infty}(t)$, $A_{\inf}(t)$ are defined by (\ref{AAA}). Then we have
   	\begin{equation}\label{graduest}
   	\norm{\nabla u(\cdot,t)}_{L^{\infty}(\mathbb{R}^{2})} \:\leq\: C(D, \gamma)\Big(1+\log_{+}\frac{A_{\gamma}(t)}{A_{\inf}(t)} \Big),
   	\end{equation}
   where $\log_{+}(x) = \max\{\log x, 0\}$.
   \end{prop}

   \begin{proof}[Proof of Proposition \ref{gradientofvelocity}]

   By the Biot-Savart law, we know that
   \[
   u(x,t) = \nabla^{\perp}\displaystyle\int_{D}G_{D}(x,y)\omega(y,t)dy.
   \]
   From Propositions \ref{estimatesofgreenfunction} and \ref{estimatesfromxiao}, it is natural to consider three cases: the inner part
$D\setminus T(r/2)$, the outer part $\mathbb{R}^{2} \setminus (T(r/2)\cup D)$ and the tubular neighborhood  $T(r)$.

   Now we analyze $u(x,t)$ in these three cases as follows. The constants in the estimates below may change from line to line.
Recall that without loss of generality $\|\omega\|_{L^\infty}=1.$

   \textbf{Case 1}: $x \in D \setminus T(r/2)$. By Proposition \ref{estimatesofgreenfunction}, $G_{D}(x,y) = \frac{1}{2\pi}\log|x-y| + h(x,y)$, thus
   \begin{equation}
   \begin{split}
   u(x,t) & =  \nabla^{\perp}\int_{\Omega(t)}G_{D}(x,y)dy\\
   &  = \nabla^{\perp}\int_{\Omega(t)}\frac{1}{2\pi}\log|x-y|dy + \nabla^{\perp}\int_{\Omega(t)}h(x,y)dy\\
   & := J_{1}(x,t) + J_{2}(x,t).
   \end{split}
   \end{equation}

   The estimate \eqref{graduest} for
 $\norm{\nabla J_{1}(\cdot,t)}_{L^{\infty}(D \setminus T(r/2))}$ has been done in Proposition 1 in \cite{Constantin}.\\

   To estimate $\norm{\nabla J_{2}(\cdot,t)}_{L^{\infty}(D \setminus T(r/2))}$, recall that $h(x,y)$ is a harmonic function solving
   \begin{equation}
   \Delta_{x}h = 0,  \quad h|_{x\in\partial D} = -\frac{1}{2\pi}\log|x-y|.
   \end{equation}

   By interior estimate of the derivatives of harmonic function (see e.g.
\cite{Evans}), we have for any $x \in D \setminus T(r/2)$, $y \in D \setminus T(r/4)$,
   \begin{equation}\label{derivativeofharmonic}
   \begin{split}
     |\nabla^{2}h(x,y)| & \leq \sup_{x \in D \setminus T(r/2)}|\nabla^{2}h(x,y)| \quad \quad \text{(fixed $y$)}\\
     & \leq Cr^{-2}\sup_{x \in \partial D}|h(x,y)|\\
     & \leq Cr^{-2}|\log(r/4)| \leq C(D),
   \end{split}
   \end{equation}
   the last inequality follows because $r$ depends only on $D$. For any $x \in D \setminus T(r/2)$ and $y \in D \cap T(r/4)$ we have,
using estimates on derivatives of harmonic functions and Proposition~\ref{estimatesofgreenfunction},
   \begin{equation}\label{derivativeofharmonic2}
   \begin{split}
   |\nabla^{2}h(x,y)| & \leq \sup_{x \in D \setminus T(r/2)}|\nabla^{2}h(x,y)| \quad \quad \text{(fixed $y$)}\\
   & \leq Cr^{-2}\sup_{x \in \partial (D \setminus T(3r/8))}|h(x,y)|\\
   & \leq \frac{1}{2\pi}Cr^{-2}|\log(r/8)| \leq C(D).
   \end{split}
   \end{equation}
   By (\ref{derivativeofharmonic}) and (\ref{derivativeofharmonic2}), we have
   \[
   \sup_{x \in D \setminus T(r/2), y \in D}|\nabla^{2}h(x,y)| \leq C(D)
   \]
for any $y \in D.$
   Therefore
   \[
   \norm{\nabla J_{2}(\cdot,t)}_{L^{\infty}(D \setminus T(r/2))} \leq C(D).
   \]

   \textbf{Case 2}: $x \in \mathbb{R}^{2} \setminus (T(r/2)\cup D)$. By Proposition \ref{estimatesofgreenfunction}, we have
   \[
   |\nabla^{2}G_{D}(x,y)| \leq C(D)|x-y|^{-2}.
   \]
   Thus
   \[
   |\nabla u(x,t)|  \leq  |\int_{\Omega(t)}\nabla^{2}G_{D}(x,y)dy| \leq C(D)r^{-2} \leq C(D)
   \]
   since $|x-y| \geq r/2$.\\

   \textbf{Case 3}: $x \in T(r/2)$. We have
   \begin{equation}
   \begin{split}
   u(x,t) & =  \nabla^{\perp}\int_{\Omega(t)}G_{D}(x,y)dy\\
   &  = \nabla^{\perp}\int_{\Omega(t)\cap T(r)}G_{D}(x,y)dy + \nabla^{\perp}\int_{\Omega(t)\cap T(r)^{c}}G_{D}(x,y)dy\\
   & := J_{3}(x,t) + J_{4}(x,t).
   \end{split}
   \end{equation}

   For $J_{4}(x,t)$, note that for $y \in \Omega(t)\cap T(r)^{c}$ and $x \in T(r/2)$, we have $|x - y| \geq r/2$. Similar to the estimate of Case 2,
we have $\norm{\nabla J_{4}(\cdot,t)}_{L^{\infty}(T(r/2))} \leq C(D)$. Now we turn to $J_{3}(x,t)$. By Proposition \ref{estimatesfromxiao}, we can rewrite $J_{3}(x,t)$ as

   \begin{equation}
   \begin{split}
   	J_{3}(x,t) & =  \frac{1}{2\pi}\int_{\Omega(t)\cap T(r)}\frac{(x-y)^{\perp}}{|x-y|^{2}}dy - \frac{1}{2\pi}\int_{\Omega(t)\cap T(r)}\frac{(x-\widetilde{y})^{\perp}}{|x-\widetilde{y}|^{2}}dy\\
   	& \quad + \int_{\Omega(t)\cap T(r)}\nabla_{x}^{\perp}B(x,y)dy\\
   	& :=  J_{31}(x,t) - J_{32}(x,t) + J_{33}(x,t).
   \end{split}
   \end{equation}

   By estimate (\ref{c2alphaestimateofB}) on $B(x,y)$, we have that $\norm{\nabla J_{33}(\cdot, t)}_{L^{\infty}(T(r/2))} \leq C(D)$.\\

   For $J_{31}$, we claim that
   \begin{equation}
   \norm{\nabla J_{31}(\cdot, t)}_{L^{\infty}(T(r/2))} \leq C(D,\gamma)\bigg(1 + \log_{+}\frac{A_{\gamma}(t)}{A_{\inf}(t)}\bigg).
   \end{equation}

Indeed, we claim that the argument of Proposition 1 in \cite{Constantin} can be used to control $J_{31}.$ The only issue one has to address is that
the region of integration in $J_{31}$ may have a corner created by intersection of $T(r)$ and $\Omega(t).$ But this difficulty is completely artificial
since $x \in T(r/2)$ is at a distance $r/2$ away from the possible location of the corner on $\partial T(r).$ We can simply smooth out the integration
region, and the error we would create by doing so is bounded from above by a constant.


 For the estimate of $J_{32}$, note that by Proposition \ref{tubularneighborhood}, $S(y)$ is a bijective $C^3$ mapping on $T(r).$
Let $F(y) = \big|\dfrac{\partial y}{\partial\widetilde{y}}\big|$ be the Jacobian of $S^{-1}$, and let $\widetilde{\Omega}_r(t)$ be the
image of $\Omega(t)\cap T(r)$ under the mapping $S(y)$ (recall $\widetilde{y}=S(y);$ we maintain this notation
redundancy for notational convenience).
It is not hard to check that $\widetilde{\Omega}_r(t) \subset T(r),$ simply by the definition of $S.$
Then we have

   \begin{equation}\label{transformedI12}
   J_{32}(x,t) = \frac{1}{2\pi}\displaystyle\int_{\widetilde{\Omega}_r(t)}\frac{(x-y)^{\perp}}{|x-y|^{2}}F(y)dy,
   \end{equation}

   where $F(y)$ is non-zero and $C^{2}$ in $T(r)$.

Since $S$ is $C^3,$ $\Omega(t)$ and $\widetilde{\Omega}_r(t)$ have the same regularity, so $J_{32}$ and $J_{31}$ are
quite similar. The only difference is that $J_{32}$ has a $C^2$ weight function $F(y)$.

   We rewrite $J_{32}$ as follows:

   \begin{equation}\label{34r}
	 J_{32}(x,t) =  \frac{1}{2\pi}\displaystyle\int_{\widetilde{\Omega}_r(t)}\frac{(x-y)^{\perp}}{|x-y|^{2}}(F(y)-F(x))dy +
\frac{1}{2\pi}\displaystyle\int_{\widetilde{\Omega}_r(t)}\frac{(x-y)^{\perp}}{|x-y|^{2}}F(x)dy.
   \end{equation}

The estimate of the second term in \eqref{34r} is identical to that of $J_{31}$ term. On the other hand, the derivatives of the first term
are bounded by constant since the expression under the integral is going to be $L^1.$


   \end{proof}

 Given the inequality (\ref{Agammaestimate}), we still need to derive a bound for $\norm{(\nabla u)w}_{\dot{C}^{\gamma}(\Omega(t))}$.  \\

   \begin{prop}\label{estimateofgradientuw}
   	Let $u$, $w$, $A_{\gamma}$, $A_{\inf}$ and $A_{\infty}$ be defined via (\ref{BSL}), (\ref{definitionOfW}) and (\ref{AAA}). Then we have
   	\begin{equation}\label{secondkey531}
   	\norm{(\nabla u)w}_{\dot{C}^{\gamma}(\Omega(t))} \leq C(D, \gamma)A_{\gamma}(t)\Big(1+\log_{+}\frac{A_{\gamma}(t)}{A_{\inf}(t)} \Big) + C(D, \gamma)A_{\infty}(t).
   	\end{equation}
   \end{prop}

   \begin{proof}[Proof of Proposition \ref{estimateofgradientuw}]
   Since $G_{D}(x,y)$ behaves differently depending on whether $x$, $y$ are close to the boundary or not, we consider two cases.

   \textbf{Case 1}: $x \in \Omega(t)\cap(D \setminus T(r/4))$. We have

   \begin{equation}
   \begin{split}
   u(x,t) & =  \nabla^{\perp}\int_{\Omega(t)}G_{D}(x,y)dy\\
   &  = \nabla^{\perp}\int_{\Omega(t)}\frac{1}{2\pi}\log|x-y|dy + \nabla^{\perp}\int_{\Omega(t)}h(x,y)dy\\
   & := J_{1}(x,t) + J_{2}(x,t).
   \end{split}
   \end{equation}

   $J_{1}$ can be regarded as the velocity generated by patch $\omega = \chi_{\Omega(t)}$ in $\mathbb{R}^{2}$. Note that by definition (\ref{definitionOfW}), $w$ is the vector field that is tangent to $\Omega(t)$. Thus by Corollary 1 in \cite{Constantin}, we have
   \[
     \norm{(\nabla J_{1})w}_{\dot{C}^{\gamma}(\Omega(t) \cap (D \setminus T(r/4)))} \leq \norm{(\nabla J_{1})w}_{\dot{C}^{\gamma}(\Omega(t))} \leq C(D, \gamma)A_{\gamma}\Big(1+\log_{+}\frac{A_{\gamma}(t)}{A_{\inf}(t)} \Big).
   \]

   To estimate $J_{2}$, note that by the argument identical to that used to derive (\ref{derivativeofharmonic})
and (\ref{derivativeofharmonic2}), we have
   \[
   \sup_{x \in D \setminus T(r/4),y \in D}|\nabla^{n}h(x,y)| \leq C(n,D).
   \]

   Then we calculate,
   \begin{equation}\label{absoluteestimate}
   \begin{aligned}
   	 \norm{(\nabla J_{2})w}_{\dot{C}^{\gamma}(\Omega(t) \cap (D \setminus T(r/4)))} & \leq \norm{\nabla J_{2}(\cdot,t)}_{L^{\infty}((D \setminus T(r/4)))}\norm{w(\cdot,t)}_{\dot{C}^{\gamma}(\Omega(t))}\\
   	& + \norm{w(\cdot,t)}_{L^{\infty}(\Omega(t))}\norm{\nabla J_{2}(\cdot,t)}_{\dot{C}^{\gamma}((D \setminus T(r/4)))}\\
   	& \leq C(D,\gamma)(A_{\gamma}(t) + A_{\infty}(t)).
   \end{aligned}
   \end{equation}

   \textbf{Case 2}: $x \in \Omega(t) \cap T(r/2)$. We write $u(x,t)$ as
   \begin{equation}\label{decompositionToI1I2}
   \begin{split}
   u(x,t) & =  \nabla^{\perp}\int_{\Omega(t)}G_{D}(x,y)dy\\
   &  = \nabla^{\perp}\int_{\Omega(t)\cap T(r)}G_{D}(x,y)dy + \nabla^{\perp}\int_{\Omega(t)\cap T(r)^{c}}G_{D}(x,y)dy\\
   & := I_{1}(x,t) + I_{2}(x,t).
   \end{split}
   \end{equation}
   For the estimate of $I_{2}(x,t)$, recall again that $G_{D}(x,y) = \frac{1}{2\pi}\log|x-y| + h(x,y)$, where $h(x,y)$ is harmonic in $x$. Note that when $y\in\Omega(t)\cap T(r)^{c}$ and $x\in T(r/2)$, we have $|x - y| \geq r/2$. So by Proposition~\ref{estimatesofgreenfunction}
we have $|\nabla^{3}G_{D}(x,y)| \leq C(D)|x-y|^{-3}$, and therefore $\norm{\nabla^{k} I_{2}(\cdot,t)}_{L^{\infty}(T(r/2))} \leq C(D)$, for $k = 1,2$. Then we obtain
   \[
     \norm{(\nabla I_{2})w}_{\dot{C}^{\gamma}(\Omega(t)\cap T(r/2)} \leq C(D,\gamma)(A_{\gamma}(t) + A_{\infty}(t)).
   \]

   Now we turn to $I_{1}(x,t)$. By Proposition \ref{estimatesfromxiao}, we can rewrite $I_{1}(x,t)$ as

   \begin{equation}\label{decompositionOfI1}
   \begin{split}
   I_{1}(x,t) & =  \frac{1}{2\pi}\int_{\Omega(t)\cap T(r)}\frac{(x-y)^{\perp}}{|x-y|^{2}}dy - \frac{1}{2\pi}\int_{\Omega(t)\cap T(r)}\frac{(x-\widetilde{y})^{\perp}}{|x-\widetilde{y}|^{2}}dy\\
   & \quad + \int_{\Omega(t)\cap T(r)}\nabla_{x}^{\perp}B(x,y)dy\\
   & :=  I_{11}(x,t) + I_{12}(x,t) + I_{13}(x,t).
   \end{split}
   \end{equation}

   First by the estimate (\ref{c2alphaestimateofB}) on $B(x,y)$, similarly to (\ref{absoluteestimate}), we have
   \[
     \norm{(\nabla I_{13})w}_{\dot{C}^{\gamma}(\Omega(t)\cap T(r/2))} \leq C(D,\gamma)(A_{\gamma}(t) + A_{\infty}(t)).
   \]

   The estimate of $\norm{(\nabla I_{11})w}_{\dot{C}^{\gamma}(\Omega(t)\cap T(r/2)}$ is the same as that of $J_{1}$, as this term can be regarded as generated by a patch in $\mathbb{R}^{2}$. Note that we still have the issue coming from the corner created by the intersection of $T(r)$ and $\Omega(t)$ in the region of integration in $I_{11}$. This difficulty is artificial since $x \in T(r/2)$ is at a distance $r/2$ away from the possible location of the corner on $\partial T(r).$ We can simply smooth out the integration
   region, and the error created by doing so is bounded from above by a constant.

   For the remaining $I_{12}$ term, similarly to \eqref{transformedI12}, we have
   \begin{equation}\label{definitionOfI12}
   I_{12}(x,t) = \frac{1}{2\pi}\displaystyle\int_{\widetilde{\Omega}(t)\cap T(r)}\frac{(x-y)^{\perp}}{|x-y|^{2}}F(y)dy,
   \end{equation}
with $F \in C^2.$ In the remainder of this section, we will demonstrate that
   \begin{equation}
   	\norm{(\nabla I_{12})w}_{\dot{C}^{\gamma}(\Omega \cap T(r/2))} \leq C(D, \gamma)A_{\gamma}(t)\Big(1+\log_{+}\frac{A_{\gamma}(t)}{A_{\inf}(t)} \Big).
   \end{equation}
   This would complete the proof. Indeed, note that the regions we consider in Cases 1 and 2 overlap, and therefore the H\"older estimate
in all $\Omega(t)$ follows by a simple argument using a bound on $\|\nabla u\|_{L^\infty}$ we proved earlier.
   \end{proof}

  All estimates we will show hold uniformly in time.
For this reason we will drop $t$ in the arguments of all functions for notational convenience.

  \begin{prop}\label{anulusestimate}
 	Let $I_{12}$, $\varphi$, $w$, $A_{\gamma}$, $A_{\inf}$ and $A_{\infty}$ be defined via (\ref{definitionOfI12}), (\ref{definitionOfVarphi}), (\ref{definitionOfW}) and (\ref{AAA}) respectively. We have
 	\[
 	\norm{(\nabla I_{12})w}_{\dot{C}^{\gamma}(\Omega \cap T(r/2))} \leq C(D, \gamma)A_{\gamma}\Big(1+\log_{+}\frac{A_{\gamma}}{A_{\inf}} \Big).
 	\]
  \end{prop}

  Before we give the proof, we need to introduce more notation. Recall that $\varphi$ and $w$ have not been defined outside
$\overline{\Omega}$.
We use Whitney-type extension theorem (see page 170 in \cite{Stein}) to extend $\varphi$ to be defined on $\mathbb{R}^{2}$ so that its
$C^{1,\gamma}$ norm increases at most by a universal factor $C(\gamma)$ depending only on $\gamma$. It is not hard to make sure that the
extension is negative outside of $\overline{\Omega}.$ We extend the definition $w(x) = (\nabla^{\perp}\varphi)(x)$ for $x
\notin \overline{\Omega}.$
Now we define $\widetilde{\varphi}$ to be
$\varphi \circ S$ and $\widetilde{w}$ to be $\nabla^{\perp}\widetilde{\varphi}$. Since the reflection $S$ is only defined on $T(r)$, so are
$\widetilde{\varphi}$ and $\widetilde{w}$. Recall the notation $\widetilde{\Omega}_r$ for $S(\Omega);$ in what follows we will omit the
subscript $r$ for notational convenience and write simply $\widetilde{\Omega}.$
Note that $\widetilde{\varphi}(x)$ vanishes on $\partial\widetilde{\Omega} \cap T(r)$, being
positive on $\widetilde{\Omega} \cap T(r)$ and non-positive elsewhere inside $T(r)$. We also have that $\widetilde{w}$ is tangent to
$\partial\widetilde{\Omega}$ inside $T(r)$.
Direct calculation shows that
  \begin{equation}\label{comparableW}
  	\widetilde{w}(x) = \text{Cofactor}(\nabla S(x)) \nabla^{\perp}\varphi(S(x))
  \end{equation}
  where $\text{Cofactor}(M)$ denotes the cofactor matrix of $M$. Since $S$ is a $C^{3}$ mapping with nonzero and finite Jacobian,
it is easy to check that
  \begin{equation}\label{ComparableWTildeX}
  	|w(\widetilde{x})| \leq C(D)|\widetilde{w}(x)|, \quad \forall x \in T(r),
  \end{equation}
  \begin{equation}\label{ComparableAgamma}
  	\norm{\widetilde{w}}_{\dot{C}^{\gamma}(T(r))} \leq C(D, \gamma) \norm{w}_{\dot{C}^{\gamma}(T(r))}.
  \end{equation}

  We let $d(x) := \text{dist}(x,\widetilde{\Omega})$, for any $x \in T(r) \setminus\widetilde{\Omega}$,  and let
$P_{x} \in \partial\widetilde{\Omega}$ be the point such that $d(x) = \text{dist}(x,P_{x})$ (if there are multiple such points,
we pick any one of them). We denote as usual $\widetilde{P}_{x}=S(P_x)$ the symmetric image of $P_{x}$ over the boundary of $D$
(see Figure \ref{fig1}).

  \begin{figure}[h]
  	\includegraphics[scale=0.5]{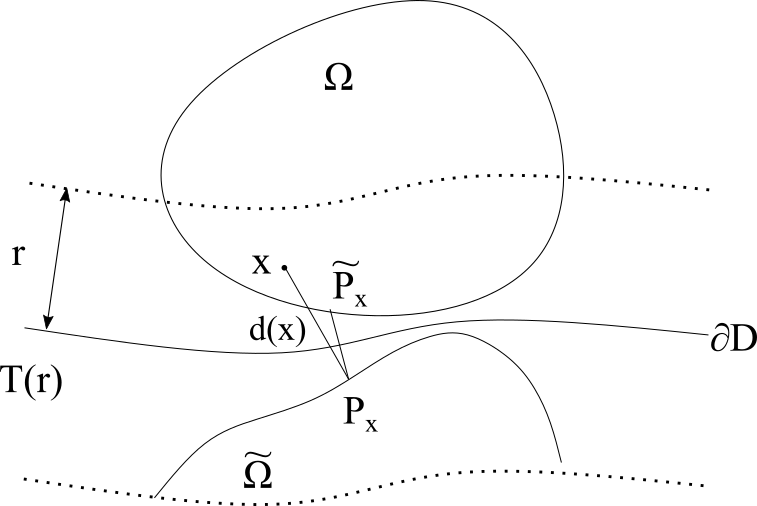}
  	\caption{}
  	\label{fig1}
  \end{figure}

  Consider any two points $x, x^{\prime} \in \Omega \cap T(r/2)$. Assume, without loss of generality, that $d(x) \leq d(x^{\prime})$.
With $g := (\nabla I_{12})w$, we have

  \begin{equation}\label{govergamma}
  \begin{split}
  \dfrac{|g(x)-g(x^{\prime})|}{|x-x^{\prime}|^{\gamma}} & \leq |\nabla I_{12}(x^{\prime})|\norm{w}_{\dot{C}^{\gamma}(\Omega \cap T(r/2))}\\
  & \quad + \dfrac{|\nabla I_{12}(x) - \nabla I_{12}(x^{\prime})|}{|x-x^{\prime}|^{\gamma}}|w(x)|.
  \end{split}
  \end{equation}

  By the argument identical to the one in Proposition \ref{gradientofvelocity}, we have that
  \[
  	\norm{\nabla I_{12}}_{L^{\infty}} \leq C(D, \gamma)\Big(1+\log_{+}\dfrac{A_{\gamma}}{A_{\inf}} \Big).
  \]

  Therefore the first term on the right hand side of (\ref{govergamma}) is bounded by
$C(D,\gamma)A_{\gamma}\Big(1+\log_{+}\dfrac{A_{\gamma}}{A_{\inf}} \Big)$. Hence it suffices to bound the second term.\\

  Note that
  \begin{equation}\label{omegapx}
  \begin{split}
  |w(x)| & \leq |w(\widetilde{P}_{x})|+|w(\widetilde{P}_{x})-w(x)|\\
  & \leq |w(\widetilde{P}_{x})| + C(\gamma)A_{\gamma}d(x)^{\gamma}.
  \end{split}
  \end{equation}
  The last inequality holds because we have $|x - \widetilde{P}_{x}| \leq 3d(x)$. Indeed by definition (see Figure \ref{fig1}),
we have $d(x) \geq \text{dist}(P_{x}, D)$, note that $\text{dist}(P_{x},D) = \text{dist}(\widetilde{P}_{x},D)$.
So we have $|P_{x} - \widetilde{P}_{x}| \leq 2d(x)$, therefore $|x - \widetilde{P}_{x}| \leq 3d(x)$.

  Now it remains to estimate $\dfrac{|\nabla I_{12}(x) - \nabla I_{12}(x^{\prime})|}{|x-x^{\prime}|^{\gamma}}$ in (\ref{govergamma}).
The following proposition provides us an appropriate estimate.

  \begin{prop}\label{T2estimate}
  	Let $I_{12}$, $A_{\gamma}$, $A_{\inf}$ and $A_{\infty}$ be defined via (\ref{definitionOfI12}) and (\ref{AAA}), and let $\widetilde{w}$, $\widetilde{\varphi}$ and $P_{x}$ be defined as in the paragraph below Proposition \ref{anulusestimate}. For $x,x^{\prime}\in \Omega \cap T(r/2)$, with $d(x) \leq d(x^{\prime})$, we have
  	\[
  	\frac{|\nabla I_{12}(x) - \nabla I_{12}(x^{\prime})|}{|x-x^{\prime}|^{\gamma}} \leq C(D,\gamma)\Big(1+\log_{+}\frac{A_{\gamma}}{A_{\inf}}\Big)\min\Big\{\frac{A_{\gamma}}{|\widetilde{w}(P_{x})|}, d(x)^{-\gamma}\Big\}.
  	\]
  \end{prop}

Let us first prove Proposition \ref{anulusestimate} by assuming Proposition \ref{T2estimate}.

  \begin{proof}[Proof of Proposition \ref{anulusestimate}]
  	By (\ref{omegapx}), we know $|w(x)| \leq |w(\widetilde{P}_{x})| + C(\gamma)A_{\gamma}d(x)^{\gamma}$. By (\ref{ComparableWTildeX}), we have that for $x \in \Omega \cap T(r)$, $|w(\widetilde{P}_{x})| \leq C(D)|\widetilde{w}(P_{x})|$. Together with Proposition \ref{T2estimate}, we have
  	\begin{tabbing}
  		$\quad\dfrac{|\nabla I_{12}(x) - \nabla I_{12}(y)}{|x-y|^{\gamma}}|w(x)|$ \= $\leq$  $C(D,\gamma)A_{\gamma}d(x)^{\gamma}\Big(1+\log_{+}\dfrac{A_{\gamma}}{A_{\inf}}\Big)\min\Big\{\dfrac{A_{\gamma}}{|\widetilde{w}(P_{x})|}, d(x)^{-\gamma}\Big\}$\\
  		\> $\quad$ $+$ $C(D,\gamma)|\widetilde{w}(P_{x})|\Big(1+\log_{+}\dfrac{A_{\gamma}}{A_{\inf}}\Big)\min\Big\{\dfrac{A_{\gamma}}{|\widetilde{w}(P_{x})|}, d(x)^{-\gamma}\Big\}$\\
  		\> $\leq$ $C(D,\gamma)A_{\gamma}\Big(1+\log_{+}\dfrac{A_{\gamma}}{A_{\inf}}\Big).$
  	\end{tabbing}
  	
  \end{proof}

 In its turn, Proposition \ref{T2estimate} is a direct corollary of the following two lemmas that will be proved below.

  \begin{lemma}\label{dgamma}
  	In the same setting as in Proposition \ref{T2estimate}, for $x,x^{\prime}\in \Omega \cap T(r/2)$, with $d(x) \leq d(x^{\prime})$, we have
  	\[
  	\frac{|\nabla I_{12}(x) - \nabla I_{12}(x^{\prime})|}{|x-x^{\prime}|^{\gamma}} \leq C(D,\gamma)d(x)^{-\gamma}.
  	\]
  \end{lemma}

  \begin{lemma}\label{logestimate}
  	In the same setting as in Proposition \ref{T2estimate}, let $r_{x} := \Big(\dfrac{|\widetilde{w}(P_{x})|}{2\widetilde{A}_{\gamma}}\Big)^{\frac{1}{\gamma}}$, where $\widetilde{A}_{\gamma} := \norm{\widetilde{w}}_{\dot{C}^{\gamma}(T(r))}$. For $x,x^{\prime}\in \Omega \cap T(r/2)$, with $d(x) \leq \min\{d(x^{\prime}), 2^{-4-1/\gamma}r_{x}\}$, we have
  	\[
  	\frac{|\nabla I_{12}(x) - \nabla I_{12}(x^{\prime})|}{|x-x^{\prime}|^{\gamma}} \leq C(D,\gamma)\Big(1+\log_{+}\frac{A_{\gamma}}{A_{\inf}}\Big)\frac{A_{\gamma}}{|\widetilde{w}(P_{x})|}.
  	\]
  \end{lemma}

Indeed, here is the proof of Proposition \ref{T2estimate} by using Lemma \ref{dgamma} and Lemma \ref{logestimate}.
 \begin{proof}[Proof of Proposition \ref{T2estimate}]
  	Recall that (\ref{ComparableAgamma}) implies
  	\begin{equation}\label{ComparableAgammaAndA}
  		\widetilde{A}_{\gamma} \leq C(D,\gamma) A_{\gamma}.
  	\end{equation}
  	
  	Due to Lemma \ref{dgamma}, we only need to consider the case where
  	\[
  	d(x) \leq \widetilde{C}^{-1}\Big(\frac{|\widetilde{w}(P_{x})|}{A_{\gamma}}\Big)^{\frac{1}{\gamma}}.
  	\]
  	We pick the constant $\widetilde{C} = 16(4C(D,\gamma))^{1/\gamma}$, where $C(D, \gamma)$ is the same as that in (\ref{ComparableAgammaAndA}).
Note that by the choice of $\widetilde{C}$, we have $d(x) \leq 2^{-4-1/\gamma}r_{x}$. Then Lemma \ref{logestimate} completes the proof.
  \end{proof}

  Now it is left to prove Lemma \ref{dgamma} and Lemma \ref{logestimate}. Let us remark that, without loss of generality,
we can always assume that $|x-x^{\prime}|$ and $d(x)$ are sufficiently small. Indeed, the estimates we are working on are of H\"older type, and if
$|x-x^{\prime}|$ exceeds some small constant that may only depend on $D$ or $\gamma$ then these estimates follow easily from the bounds similar
to the one on $\|\nabla u\|_{L^\infty}$ we already established. On the other hand, if $d(x) \geq C>0,$ then in \eqref{definitionOfI12} we have
$|x-y| \geq C$ for all $y$. Then $I_{12}$ will be smooth, and the estimates on $\nabla I_{12}$ can be obtained without any effort.

We begin with an auxiliary claim that will be used in the
proofs of Lemma \ref{dgamma} and Lemma \ref{logestimate}. Consider any $x, x' \in \Omega\cap T(r)$, with $|x-x'| \leq r,$ $d(x) \leq d(x')$.
Given any point $z,$ denote $Q_z$ the point on $\partial D$ closest to $z.$  Let $x''$
be such that $x, x', x''$ form an equilateral triangle. Obviously, there are two possible choices of $x''$ and we choose the point which is further
away from a line tangent to $\partial D$ passing through $Q_x.$
  \begin{lemma}\label{xx'x''Estimate}
  Consider the points $x,$ $x'$ and $x''$ as above.
	 We parametrize the segments $[xx^{\prime\prime}]$ and $[x^{\prime}x^{\prime\prime}]$ by
	\[
	z_{1}(s) = x + s(x^{\prime\prime} - x),
	\]
	\[
	z_{2}(s) = x^{\prime} + s(x^{\prime\prime} - x^{\prime}),
	\]
$0 \leq s \leq 1.$
	Then there exists a universal constant $C > 0$, such that
	\begin{equation}\label{xx''Estimate}
	d(z_{1}(s)) \geq \frac13 \max\{d(x), s|x-x^{\prime}|\},
	\end{equation}
	\begin{equation}\label{x'x''Estimate}
	d(z_{2}(s)) \geq \frac13 \max\{d(x), s|x-x^{\prime}|\}.
	\end{equation}
  \end{lemma}

  \begin{proof}[Proof of Lemma \ref{xx'x''Estimate}]
	
 Let us consider (\ref{xx''Estimate}). 
 Observe that for any $z \in \Omega\cap T(r)$, we have
	\begin{equation}\label{d(z)d_zEquivalent}
	d_{z} \leq d(z) \leq 2d_{z}.
	\end{equation}
	The second inequality is due to $\widetilde{z} \in \widetilde{\Omega}$ and $\text{dist}(z, \widetilde{z}) = 2d_{z}$.
	
Choose local coordinates $(p_1,p_2)$ with center at $Q_x$ and $p_1$ directed along the tangent at $Q_x$ and towards $Q_{x'}$ (if $Q_{x'} = Q_x$ then the estimate is immediate).
Denote $\beta$ the angle between $p_2$ axis and $xx''$ directed interval.
Due to our choice of $\epsilon$ in \eqref{UnitTangentVectorDifference}, definition of $r(\epsilon)$ and the choice of $r \leq r(\epsilon,$ elementary geometric considerations show
that $\cos \beta \geq \frac12.$ Therefore, the second coordinate of $z_1(s)$ satisfies $(z_1(s))_2 \geq d_x + \frac12 s |x-x'|.$ Using again the control over $\partial D$
over scale $r$ afforded by our choice of $\epsilon$ and \eqref{d(z)d_zEquivalent}, it is not hard to pass from the last estimate to \eqref{xx''Estimate}.

The case of \eqref{x'x''Estimate} is similar. We leave details to the interested reader.

  \end{proof}

 Now we prove Lemma \ref{dgamma}.

  \begin{proof}[Proof of Lemma \ref{dgamma}]
  	We consider two cases.
  	
  	\textbf{Case 1}: $|x-x^{\prime}| \leq d(x)$. By mean value theorem, for any $z, z^{\prime}$ such that the segment $[zz^{\prime}]$ is at a positive distance from $\widetilde{\Omega}$, we have
  	\[
  	\frac{|\nabla I_{12}(z) - \nabla I_{12}(z^{\prime})|}{|z-z^{\prime}|^{\gamma}} \leq |\nabla^{2} I_{12}(Z_{xx^{\prime}})||z-z^{\prime}|^{1-\gamma},
  	\]
  	for some point $Z_{zz^{\prime}}$ on the segment $[zz^{\prime}]$.
  	
  	Note that for any point $Z \not\in\widetilde{\Omega}$, we have
  	\begin{equation}\label{doubleGradientOfI12}
  	|\nabla^{2} I_{12}(Z)| \leq \int_{\mathbb{R}^{2}\setminus B_{d(Z)}(Z)}\frac{C(D)}{|Z-x|^{3}}dx \leq C(D)d(Z)^{-1}.
  	\end{equation}

  	Recall that if we pick $x^{\prime\prime}$ as we did in Lemma \ref{xx'x''Estimate}, then we have $d(Z_{xx^{\prime}}), d(Z_{x^\prime x^{\prime\prime}}) \geq \frac12 d(x)$.
  	
  	Then we have
  	\[
  	\frac{|\nabla\ I_{12}(x) - \nabla I_{12}(x^{\prime\prime})|}{|x-x^{\prime\prime}|^{\gamma}} \leq C(D)d(Z_{xx^{\prime\prime}})^{-1}|x-x^{\prime\prime}|^{1-\gamma} \leq C(D)d(x)^{-1}|x-x^{\prime\prime}|^{1-\gamma} \leq C(D)d(x)^{-\gamma},
  	\]
  	\[
  	\frac{|\nabla\ I_{12}(x^{\prime}) - \nabla I_{12}(x^{\prime\prime})|}{|x^{\prime}-x^{\prime\prime}|^{\gamma}} \leq C(D)d(Z_{x^{\prime}x^{\prime\prime}})^{-1}|x^{\prime}-x^{\prime\prime}|^{1-\gamma} \leq C(D)d(x)^{-1}|x^{\prime}-x^{\prime\prime}|^{1-\gamma} \leq C(D)d(x)^{-\gamma},
  	\]
  	The last inequalities hold true because we have $|x-x^{\prime}| \leq d(x)$.
  	
  	Putting them together, we have
  	\[
  		\frac{|\nabla\ I_{12}(x) - \nabla I_{12}(x^{\prime})|}{|x-x^{\prime}|^{\gamma}} \leq C(D)d(x)^{-\gamma}.
  	\]

  	\textbf{Case 2}: $|x-x^{\prime}| \geq d(x)$. Recall that if we define $x^{\prime\prime}$, $z_1(s)$ and $z_2(s)$ as in Lemma \ref{xx'x''Estimate}, then we have
  	\[
  		d(z_{1}(s)) \geq \frac12 \max\{d(x), s|x-x^{\prime}|\},
	\]
	\[
  		d(z_{2}(s)) \geq \frac12 \max\{d(x), s|x-x^{\prime}|\}.
  	\]
    Integrating along the path $x \rightarrow x^{\prime\prime} \rightarrow x^{\prime}$ yields
  	\begin{tabbing}
  		$\qquad |\nabla I_{12}(x) - \nabla I_{12}(x^{\prime})|$ \= $\leq$ \= $\displaystyle\int_{0}^{1}|\nabla^{2} I_{12}(x+s(x^{\prime\prime}-x))||x-x^{\prime}|ds$\\
  		\> \>$+$ $\displaystyle\int_{0}^{1}|\nabla^{2} I_{12}(x^{\prime}+s(x^{\prime\prime}-x^{\prime}))||x-x^{\prime}|ds$\\
  		\> $\leq$ \> $C(D)|x-x^{\prime}|\left(\displaystyle\int_{0}^{\frac{d(x)}{|x-x^{\prime}|}}d(x)^{-1}ds + \displaystyle\int_{\frac{d(x)}{|x-x^{\prime}|}}^{1}(s|x-x^{\prime}|)^{-1}ds\right)$\\
  		\>$\leq$ \> $C(D)\left(1+\log\dfrac{|x-x^{\prime}|}{d(x)}\right).$\\
  	\end{tabbing}

	Now we have
	\[
	  \frac{|\nabla I_{12}(x) - \nabla I_{12}(x^{\prime})|}{|x-x^{\prime}|^{\gamma}} \leq C(D,\gamma)\left(1+\log\frac{|x-x^{\prime}|}{d}\right)|x-x^{\prime}|^{-\gamma} \leq C(D,\gamma)d^{-\gamma}.
	\]
	
	The last inequality follows that $1 + \log a \leq \frac{1}{\gamma}a^{\gamma}$, for $a \geq 1$.

  \end{proof}

  Now we prove Lemma \ref{logestimate} by using Lemma \ref{secondgradient} below, which we will prove later.

  \begin{lemma}\label{secondgradient}
  	Let $r_{x} := \Big(\dfrac{|\widetilde{w}(P_{x})|}{2\widetilde{A}_{\gamma}}\Big)^{\frac{1}{\gamma}}$, where $\widetilde{A}_{\gamma} := \norm{\widetilde{w}}_{\dot{C}^{\gamma}(T(r))}$.
  	For any $x \in T(r/2) \setminus \widetilde{\Omega}$, such that $d(x) \leq \dfrac{1}{4}r_{x}$ and $P_{x} \in T(r/4)$, we have
  	\[
  	|\nabla^{2} I_{12}(x)| \leq C(D,\gamma)d(x)^{-1+\gamma}r_{x}^{-\gamma}.
  	\]
  \end{lemma}

  \begin{proof}[Proof of Lemma \ref{logestimate}]
  	We consider two cases.
  	
  	\textbf{Case 1}: $|x-x^{\prime}| \geq 2^{-4-1/\gamma}r_{x}$. In this case, we have $|x-x^{\prime}|^{-\gamma} \leq C(\gamma)\dfrac{\widetilde{A}_{\gamma}}{|\widetilde{w}(P_{x})|}$. Note that by (\ref{ComparableAgamma}), we have that $\widetilde{A}_{\gamma} \leq C(D,\gamma)A_{\gamma}$. The proof follows directly from
  	\[
  	|\nabla I_{12}(x) - \nabla I_{12}(x^{\prime})| \leq 2C(D,\gamma)\Big(1+\log_{+}\frac{A_{\gamma}}{A_{\inf}} \Big),
  	\]
  	This bound can be derived by the same way as in Proposition \ref{gradientofvelocity}.
  	
  \textbf{Case 2}: $|x-x^{\prime}| < 2^{-4-1/\gamma}r_{x}$. We define $x^{\prime\prime}$, $z_1(s)$ and $z_2(s)$ as we did in Lemma \ref{xx'x''Estimate}, then we have
  \begin{equation}\label{dziEstimate}
    d(z_{i}(s)) \geq \frac12 s|x-x^{\prime}|,
  \end{equation}
  for $i = 1,2$ and $s \in [0,1]$.

  Notice that
  \begin{equation}\label{zis-Px}
  |z_{i}(s)-P_{x}| \leq |z_{i}(s) -x| + d(x) \leq 2|x-x^{\prime}|+d(x).
  \end{equation}

  So (\ref{zis-Px}) and the facts that $|x-x^{\prime}| \leq 2^{-4-1/\gamma}r_{x}$, $d(x) \leq 2^{-4-1/\gamma}r_{x}$ give us
  \begin{equation}\label{dvsrx}
  d(z_{i}(s)) \leq 2^{-2-1/\gamma}r_{x}.
  \end{equation}

  These inequalities imply
  \[
  P_{z_{i}(s)} \in B_{x} := B_{r_{x}}(P_{x}).
  \]

  Note that we have
  \[
  |\widetilde{w}(P_{z_{i}(s)}) - \widetilde{w}(P_{x})| \leq \widetilde{A}_{\gamma}|P_{z_{i}(s)}-P_{x}|^{\gamma} \leq \frac{|\widetilde{w}(P_{x})|}{2}.
  \]
  Then it implies that
  \[
  |\widetilde{w}(P_{z_{i}(s)})| \geq \frac{|\widetilde{w}(P_{x})|}{2},
  \]
  yielding
  \[
  r_{z_{i}(s)} \geq 2^{-\frac{1}{\gamma}}r_{x}.
  \]
  From (\ref{dvsrx}) it follows that
  \[
  d(z_{i}(s)) \leq \frac{1}{4}r_{z_{i}(s)}.
  \]

  To apply Lemma \ref{secondgradient}, we also need to verify $P_{z_{i}(s)} \in T(r/4)$. Without loss of generality, we can assume $|x-x^{\prime}| \leq r/32$ and $d(x) \leq r/8$. So we have
  \[
  	d(z_{i}(s)) \leq |z_{i}(s) - P_{x}| \leq |z_{i}(s) - x| + |x - P_{x}| \leq r/4.
  \]

  Now we can apply Lemma \ref{secondgradient} to $z_{i}(s)$ and get
  \[
  |\nabla^{2} I_{12}(z_{i}(s))| \leq C(D,\gamma)d(z_{i}(s))^{-1+\gamma}r_{z_{i}(s)}^{-\gamma} \leq C(D,\gamma)(s|x-x^{\prime}|)^{-1+\gamma}r_{x}^{-\gamma}.
  \]

  Integrating along the path $x \rightarrow x^{\prime\prime} \rightarrow x^{\prime}$ gives us
  \begin{tabbing}
  	$\qquad \dfrac{|\nabla I_{12}(x) - \nabla I_{12}(x^{\prime})|}{|x-x^{\prime}|^{\gamma}}$ \= $\leq$ \= $\displaystyle\int_{0}^{1}|\nabla^{2} I_{12}(x+s(x^{\prime\prime}-x))||x-x^{\prime}|^{1-\gamma} ds$\\
  	\> \>$+$ $\displaystyle\int_{0}^{1}|\nabla^{2} I_{12}(x^{\prime}+s(x^{\prime\prime}-x^{\prime}))||x-x^{\prime}|^{1-\gamma}ds$\\
  	\> $\leq$ \> $C(D,\gamma)|x-x^{\prime}|^{1-\gamma}\displaystyle\int_{0}^{1}(s|x-x^{\prime}|)^{-1+\gamma}r_{x}^{-\gamma}ds$\\
  	\> $\leq$ \> $C(D,\gamma)r_{x}^{-\gamma}$\\
  	\> $\leq$ \> $C(D,\gamma)\dfrac{A_{\gamma}}{|\widetilde{w}(P_{x})|}.$
  \end{tabbing}

  \end{proof}

  \begin{figure}[h]
  	\includegraphics[scale=0.5]{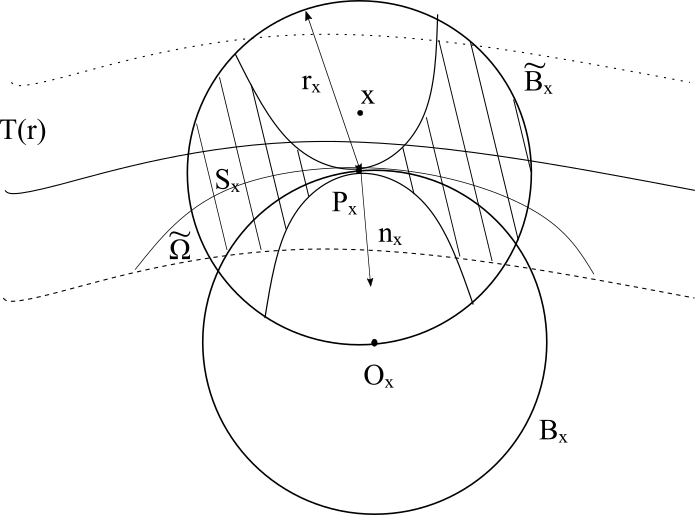}
  	\caption{}
  	\label{fig2}
  \end{figure}

  It remains to prove Lemma \ref{secondgradient}. First we need a result from \cite{KRYZ}, which is in its turn similar to the Geometric Lemma of \cite{Constantin}.

  \begin{lemma}\label{geometriclemma}
  	For any $x \in T(r)\cap D$, with $P_{x} \in T(r/4)$, let $n_{x} := \nabla\widetilde{\varphi}(P_{x})/|\nabla\widetilde{\varphi}(P_{x})|$ and $r_{x} := \Big(\dfrac{|\widetilde{w}(P_{x})|}{2\widetilde{A}_{\gamma}}\Big)^{\frac{1}{\gamma}}$. Define
  	\begin{equation}\label{Sset}
  	S_{x} := \{P_{x}+\rho\nu : \rho \in [0,r_{x}), |\nu| = 1, \left(\dfrac{\rho}{r_{x}}\right)^{\gamma} \geq 2|\nu\cdot n_{x}|, P_{x} + \rho\nu \in T(r)\}.
  	\end{equation}
  	If $\nu$ is a unit vector and $\rho \in [0,r_{x})$, then the following statements hold. \\
  1. If $\nu\cdot n_{x} \geq 0$ and $P_{x} + \rho\nu \not\in S_{x}$, then $P_{x} + \rho\nu \in \widetilde{\Omega} \cap T(r)$; \\
  2. If $\nu\cdot n_{x} \leq 0$ and $P_{x} + \rho\nu \not\in S_{x}$, then $P_{x} + \rho\nu \in T(r) \setminus \widetilde{\Omega}$ (see Figure \ref{fig2}).
  \end{lemma}

  Lemma \ref{geometriclemma} is slightly different from Lemma 3.7 in \cite{KRYZ} in that $D$ is a general smooth bounded domain. However the proofs are virtually identical. For the sake
 of completeness we present the original proof here.
  \begin{proof}[Proof of Lemma \ref{geometriclemma}]
  	
  	We only need to prove the first statement, as the proof of the second statement is analogous. Let us assume $\nu \cdot n_{x} \geq 0$ and $P_{x} + \rho\nu \notin \widetilde{\Omega} \cap T(r)$, with $|\nu| = 1$ and $\rho \geq 0$. Then we have,
  	\[
  		\nabla \widetilde{\varphi}(P_{x}) \cdot \nu \geq 0 \qquad \text{and} \qquad \widetilde{\varphi}(P_{x} + \rho\nu) \leq 0.
  	\]
  	So we must have $\nabla \widetilde{\varphi}(P_{x}) \cdot \nu \leq \widetilde{A}_{\gamma}\rho^{\gamma}$ due to $C^\gamma$ estimate on
  $\nabla \widetilde{\varphi}$ and because $\widetilde{\varphi}(P_{x}) = 0$. Thus
  	\[
  		2\nu \cdot n_{x} \leq \dfrac{2\widetilde{A}_{\gamma}\rho^{\gamma}}{|\nabla \widetilde{\varphi}(P_{x})|} = \Big(\dfrac{\rho}{r_{x}}\Big)^{\gamma},
  	\]
  	so either $\rho \geq r_{x}$ or $P_{x} + \rho\nu \in S_{x}$.
  \end{proof}

  Now, let us prove Lemma \ref{secondgradient}.

  \begin{figure}[h]
  	\includegraphics[scale=0.5]{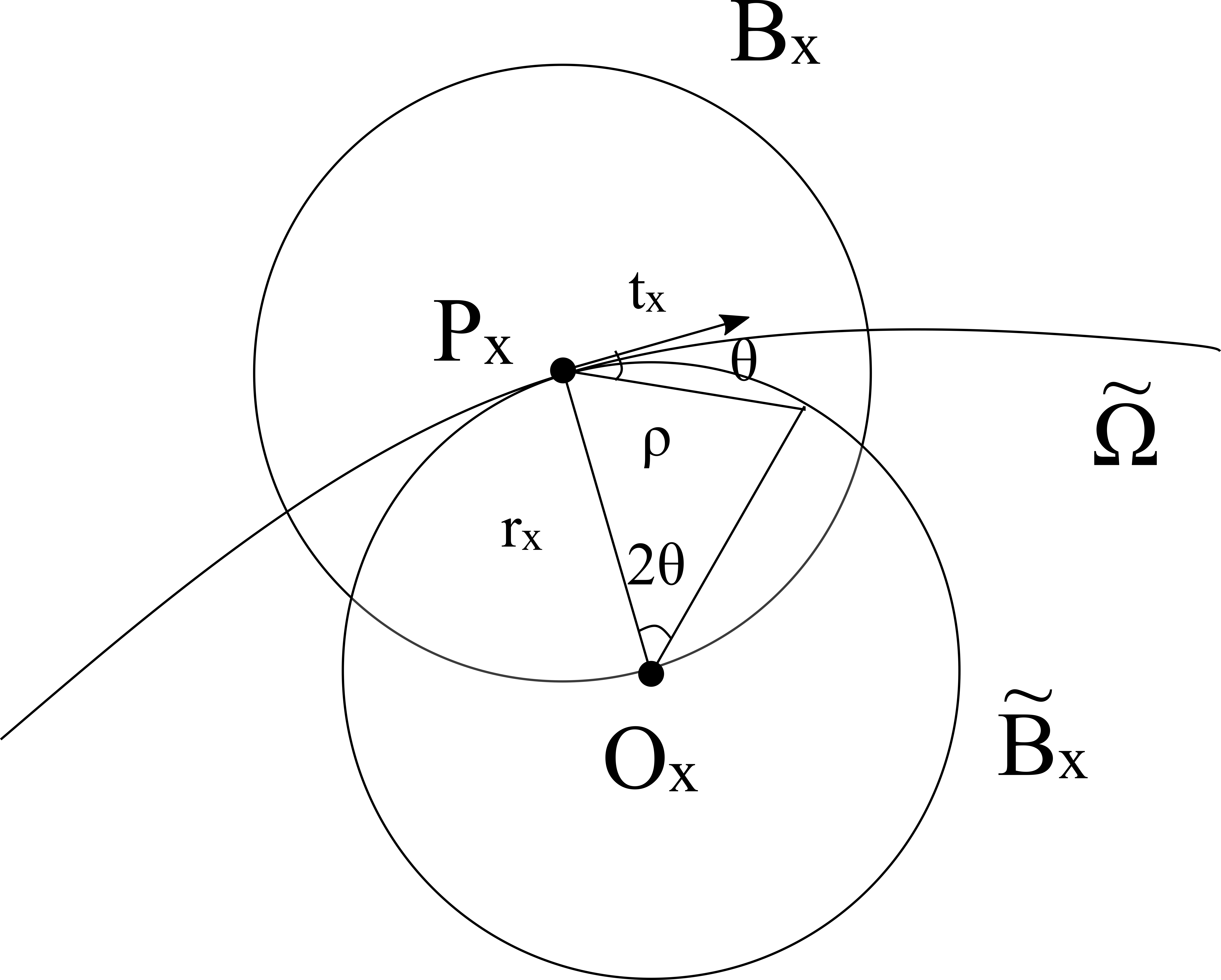}
  	\caption{}
  	\label{fig4}
  \end{figure}

  \begin{proof}[Proof of Lemma \ref{secondgradient}]
  	Let $n_{x}, S_{x},$ $\nu$ be the same as in Lemma \ref{geometriclemma} and let $B_{x} := B_{r_{x}}(P_{x})$ and $\widetilde{B}_{x} := B_{r_{x}}(O_{x})$, where $O_{x} := P_{x} + r_{x}n_{x}$ (see Figure \ref{fig2}). Then $P_{x} \in \partial\widetilde{B}_{x}$ and the unit inner normal vector to $\partial\widetilde{B}_{x}$ at $P_{x}$ is $n_{x}$.
  	
  	First we claim that,
  	\[
  		 \partial\widetilde{B}_{x} \cap B_{x} \cap T(r) \subseteq S_{x}.
  	\]
  	
  	Indeed, let $t_{x}$ be the vector tangent to $\partial\widetilde{B}_{x}$ at $P_{x}$ and let $\theta \in [0,\dfrac{\pi}{2}]$ denote the angle between $t_{x}$ and $\nu$ (see Figure \ref{fig4}). Note that $P_{x} + \rho\nu \in \partial\widetilde{B}_{x} \cap B_{x} \cap T(r)$ implies that $\theta \leq \dfrac{\pi}{6}$. By the law of sines, we have
  	\[
  	\frac{\rho}{\sin 2\theta} = \frac{r_{x}}{\cos\theta}.
  	\]
  	Since $\nu\cdot n_{x} = \sin\theta$, $\Big(\dfrac{\rho}{r_{x}}\Big)^{\gamma} \geq 2|\nu\cdot n_{x}|$ follows immediately from the fact that when $\theta \leq \dfrac{\pi}{6}$, $(2\sin\theta)^{\gamma} \geq 2\sin\theta$.\\
  	
  	Together with Lemma \ref{geometriclemma}, we have
  	\[
  	(\widetilde{\Omega}\Delta\widetilde{B}_{x}) \cap B_{x} \cap T(r) \subseteq S_{x}.
  	\]
  	
  	Define
  	\[
  	u_{\widetilde{B}_{x}}(z) := \frac{1}{2\pi} \int_{\widetilde{B}_{x}}\frac{(z-y)^{\perp}}{|z-y|^{2}}F(y)dy,
  	\]
  	recall that $F(y)$ is the Jacobian of the reflection map $S(y)$.
  	
  	We claim that
  	\begin{equation}\label{magicsymmetric}
  	|\nabla^{2}u_{\widetilde{B}_{x}}(x)| \leq \frac{C(D)}{r_{x}}.
  	\end{equation}
  	
  	We will prove (\ref{magicsymmetric}) in Lemma \ref{symmetricmagic} later.

  	By (\ref{definitionOfI12}) and definition of $u_{\widetilde{B}_{x}}$, we have
  	\begin{equation}\label{differencedouble}
  	\begin{split}
  	|\nabla^{2}I_{12}(x) - \nabla^{2}u_{\widetilde{B}_{x}}(x)|
  	&\leq \int_{T(r) \setminus B_{x}}\frac{C(D)}{|x-y|^{3}}dy\\
  	& \quad + \int_{(\widetilde{\Omega}\Delta\widetilde{B}_{x}) \cap B_{x} \cap T(r)}\frac{C(D)}{|x-y|^{3}}dy\\
  	& \quad + \int_{\widetilde{B}_{x} \cap T(r)^{c}}\frac{C(D)}{|x-y|^{3}}dy.
  	\end{split}
  	\end{equation}
  	
  	To bound the first term, note that if $y \in B^{c}_{x}$, we have $|x-y| \geq \dfrac{3}{4}r_{x}$, due to $d(x) \leq \dfrac{1}{4}r_{x}$. Therefore the first term in the right hand side can be bounded by $C(D)r_{x}^{-1}$.

	For the second term, we claim that,
  	\begin{equation}\label{Sdx}
  	\text{dist}(x,S_{x}) \geq \frac{d(x)}{2}.
  	\end{equation}
  	Indeed, if $P_{x} + \rho\nu \in B_{d(x)/2}(x) \cap T(r)$, we have $|\nu \cdot n_{x}| > \dfrac{1}{2}$ and $\rho \leq \dfrac{3}{2}d(x) < r_{x}$, hence $P_{x} + \rho\nu \not\in S_{x}$ by definition in Lemma \ref{geometriclemma}.
  	
  	Also we note that, if $|P_{x} - y| \geq 2d(x)$, we have
  	\begin{equation}\label{dcomp351}
  	|P_{x} - y| \leq |x - y| + d(x) \leq 2|x-y|.
  	\end{equation}
  	
  	Denoting by $II$ the second term in (\ref{differencedouble}), we have that
  	
  	\begin{tabbing}
  		$\qquad$ $\qquad$ $II$ \= $\leq$ $\displaystyle\int_{S_{x}}\frac{C(D)}{|x-y|^{3}}dy$\\
  		\> $\leq$ $\displaystyle\int_{S_{x}\setminus B_{2d(x)}(P_{x})}\frac{C(D)}{|x-y|^{3}}dy + C\left(\frac{d(x)}{2}\right)^{-3}|S_{x}\cap B_{2d(x)}(P_{x})|$\\
  		\> $\leq$ $\displaystyle\int_{S_{x}\setminus B_{2d(x)}(P_{x})}\frac{C(D)}{|P_{x}-y|^{3}}dy + \frac{C(D)}{d(x)^{3}}|S_{x}\cap B_{2d(x)}(P_{x})|$\\
  		\> $\leq$ $\displaystyle\int_{2d(x)}^{r_{x}}\frac{C(D)}{\rho^{3}}\Big(\frac{\rho}{r_{x}}\Big)^{\gamma}\rho d\rho + \frac{C(D)}{d(x)^{3}}\displaystyle\int_{0}^{2d(x)}\Big(\frac{\rho}{r_{x}}\Big)^{\gamma}\rho d\rho$\\
  		\> $\leq$ $C(D,\gamma)d(x)^{-1+\gamma}r_{x}^{-\gamma}.$
  	\end{tabbing}
Here we used \eqref{Sdx} in the first step, \eqref{dcomp351} in the second step, and \eqref{Sset} in the third step.

  	Now we consider the third term in (\ref{differencedouble}). Suppose $r_{x} \leq r/4$, so that $\widetilde{B}_{x} \subset T(r)$. Then the region of the integration is empty and the third term vanishes. Othervise, if $r_{x} \geq r/4$, we can redefine $S_{x}$, $B_{x}$, $\widetilde{B}_{x}$ by replacing $r_{x}$ with $r/4$ and all the arguments remain true. The reason behind this is that Lemma \ref{geometriclemma} tells us that $\partial\widetilde{\Omega}\cap B_{x} \subset S_{x}$. If we shrink both $B_{x}$ and $S_{x}$ to by the same factor, the inclusion still holds.

  	Combining with (\ref{magicsymmetric}) and (\ref{differencedouble}), we obtain
  	\[
  	|\nabla^{2}I_{12}(x)| \leq C(D,\gamma)d(x)^{-1+\gamma}r_{x}^{-\gamma} + C(D)r_{x}^{-1}.
  	\]
  	
    The result follows from $d(x) \leq \dfrac{1}{4}r_{x}$.
  	
  \end{proof}

  Next, we are going to prove (\ref{magicsymmetric}). To make the argument simpler, we let $B_{r} = \{z \in \mathbb{R}^{2}:|z-(0,-r)| < r\}$, $f$ be a function such that $|f|$, $|\nabla f|$ and $|\nabla^{2}f|$ are bounded by a universal constant $C$ in $B_{r}$. Define
  \begin{equation}\label{definitionOfu}
  u(x) = \dfrac{1}{2\pi}\displaystyle\int_{B_{r}}\frac{(x-y)^{\perp}}{|x-y|^{2}}f(y)dy.
  \end{equation}
  By setting $B_{r}$ to be $\widetilde{B}_{x}$ and $f(y)$ to be $F(y)$, (\ref{magicsymmetric}) can be derived from the following Lemma.

  \begin{lemma}\label{symmetricmagic}
  	Let $x = (0,h)$, where $h \leq 1$ is any positive real number. For $h \leq r \leq 1$, the vector field $u$ defined by (\ref{definitionOfu}) satisfies
  	\[
  		|\nabla^{2}u(x)| \leq \dfrac{C}{r}.
  	\]
  \end{lemma}

  \begin{proof}[Proof of Lemma \ref{symmetricmagic}]
  	
  	First, by assumption, for any $y$ in $B_{r}$, we have
  	\[
  	|f(y) - f(0) - \nabla f(0)\cdot y| \leq C|y|^{2}.
  	\]
  	By (\ref{definitionOfu}), we write $u(x)$ as
  	\[
  	u(x) = \dfrac{1}{2\pi}\int_{B_{r}}\frac{(x-y)^{\perp}}{|x-y|^{2}}f(y)dy = \frac{1}{2\pi}\left(K_{1} + K_{2} + K_{3}\right),
  	\]
  	where
  	\begin{tabbing}
  		$\qquad$ \= $K_{1} := \displaystyle\int_{B_{r}}\dfrac{(x-y)^{\perp}}{|x-y|^{2}}f(0)dy,$\\
  		\> $K_{2} := \displaystyle\int_{B_{r}}\dfrac{(x-y)^{\perp}}{|x-y|^{2}}\nabla f(0)\cdot \vec{y}dy,$\\
  		\> $K_{3} := \displaystyle\int_{B_{r}}\dfrac{(x-y)^{\perp}}{|x-y|^{2}}(f(y) - f(0) - \nabla f(0)\cdot \vec{y})dy.$
  	\end{tabbing}
  	
  	It suffices to prove $|\nabla^{2}K_{i}(x)| \leq \dfrac{C}{r}$, for $i = 1,2,3$.\\
  	
  	For $i = 1$, we use the argument from \cite{KRYZ}. Observe that
  	\[
  	K_{1}(x) = f(0)(\nabla^{\perp}\Delta^{-1}\chi_{B_{r}})(x).
  	\]
  	Since $|x - (0,-r)| > r$, we have by the rotational invariance of $K_{1}(x)$ (and with $n$ being the outer unit normal vector to $\partial B_{|x-(0,-r)|}\big((0,-r)\big)$)
  	
  	\begin{equation}\label{estimateI1}
  	\begin{split}
  	K_{1}(x) & = \dfrac{\big(x - (0,-r)\big)^{\perp}}{|x-(0,-r)|}|K_{1}(x)|\\
  	& = \dfrac{\big(x-(0,-r)\big)^{\perp}}{|x-(0,-r)|}\displaystyle\fint_{\partial B_{|x-(0,-r)|}\big((0,-r)\big)}n\cdot f(0)\nabla\Delta^{-1}\chi_{B_{r}}d\sigma\\
  	& = \dfrac{\big(x-(0,-r)\big)^{\perp}}{|x-(0,-r)|^{2}}\dfrac{1}{2\pi}\displaystyle\int_{B_{|x-(0,-r)|}\big((0,-r)\big)}f(0)\chi_{B_{r}}(y)dy\\
  	& = \dfrac{1}{2}f(0)r^{2}\dfrac{(x - (0,-r))^{\perp}}{|x-(0,-r)|^{2}}.
  	\end{split}
  	\end{equation}
  	
  	Differentiate this, then we have
  	\[
  	|\nabla^{2}K_{1}(x)| \leq \dfrac{C}{r}.
  	\]
  	
  	For $i = 3$, since $|x-y| \geq |y|$, we have
  	\[
  	|\nabla^{2}K_{3}(x)| \lesssim \int_{B_{r}}\frac{|y|^{2}}{|x-y|^{3}}dy \leq \int_{B_{r}}\frac{1}{|y|}dy \lesssim r.
  	\]
  	
  	For $i = 2$, first note that it suffices to control $|\nabla^{2}(k\cdot K_{2}(x))|$, for any constant unit vector $k = (k_{1}, k_{2})$. Denoting $a = \nabla f(0)$, we have
  	
  	\begin{tabbing}
  		$\quad$ $2k\cdot K_{2}(x)$ \= $=$ $-2\displaystyle\int_{B_{r}}\dfrac{(x-y)\cdot k^{\perp}}{|x-y|^{2}}(a\cdot y)dy$\\
  		\> $=$ $\displaystyle\int_{B_{r}}\nabla_{y}\cdot\big((\ln|x-y|^{2})k^{\perp}\big)(a\cdot y)dy$\\
  		\> $=$ $\displaystyle\int_{\partial B_{r}}(n\cdot k^{\perp})\ln|x-y|^{2}(a\cdot y)dS(y)$ $-$ $(k^{\perp}\cdot a)\displaystyle\int_{B_{r}}\ln|x-y|^{2}dy$\\
  		\> $:=$ $K_{21}(x) - K_{22}(x).$
  	\end{tabbing}
  	Here $n$ is the outer normal vector for $\partial B_{r}$. \\
  	
  	Controlling $\nabla^{2}K_{22}(x)$ is straightforward. Indeed, $\nabla K_{22}(x)$ is the velocity field generated by the vorticity patch $2(k^{\perp}\cdot a)\chi_{B_{r}}(x)$. By estimate (\ref{estimateI1}), we have
  	\[
  	\nabla K_{22}(x) = 2(k^{\perp}\cdot a)\pi r^{2}\dfrac{(x - (0,-r))^{\perp}}{|x-(0,-r)|^{2}}.
  	\]
  	So $|\nabla ^{2}K_{22}(x)| \leq C$, where $C$ is a universal constant.\\
  	
  	For $K_{21}(x)$, note that $\partial_{22}K_{21}(x) = -\partial_{11}K_{21}(x)$ and $\partial_{12}K_{21}(x) = \partial_{21}K_{21}(x)$, so it suffices to consider two cases.
  	
  	\begin{equation}\label{partial11}
  	\partial_{11}K_{21}(x) = 2\displaystyle\int_{\partial B_{r}}(n\cdot k^{\perp}) \dfrac{(h-y_{2})^{2} - y_{1}^{2}}{|x-y|^{4}}(a_{1}y_{1} + a_{2}y_{2})dS(y),
  	\end{equation}
  	\begin{equation}\label{partial12}
  	\partial_{12}K_{21}(x) = 4\int_{\partial B_{r}}(n\cdot k^{\perp}) \dfrac{y_{1} (h-y_{2})}{|x-y|^{4}}(a_{1}y_{1} + a_{2}y_{2})dS(y).
  	\end{equation}

  	First note that when $|y| \leq r/4$, we have $y_{2} \leq \dfrac{C}{r} y_{1}^{2}$ and that $|n\cdot k - k_{2}| \leq \dfrac{C}{r}|y_{1}|$.\\
  	
  	For (\ref{partial11}), we have, with a constant $C < \infty$ that depends on $a$ and $k$:
  	\begin{align*}
  	\partial_{11}K_{21}(x) & =  2\displaystyle\int_{\partial B_{r}\cap \{|y| \geq r/4\}}(n\cdot k^{\perp}) \dfrac{(h-y_{2})^{2} - y_{1}^{2}}{|x-y|^{4}}(a_{1}y_{1} + a_{2}y_{2})dS(y)\\
  	& \quad  + 2\displaystyle\int_{\partial B_{r}\cap\{|y| < r/4\}}(n\cdot k^{\perp}) \dfrac{(h-y_{2})^{2} - y_{1}^{2}}{|x-y|^{4}}a_{1}y_{1}dS(y)\\
  	& \quad + 2\displaystyle\int_{\partial B_{r}\cap\{|y| < r/4\}}(n\cdot k^{\perp}) \dfrac{(h-y_{2})^{2} - y_{1}^{2}}{|x-y|^{4}}a_{2}y_{2}dS(y)\\
  	& \leq C + 2\displaystyle\int_{\partial B_{r}\cap\{|y| < r/4\}}(n\cdot k^{\perp} - k_{2}) \dfrac{(h-y_{2})^{2} - y_{1}^{2}}{|x-y|^{4}}a_{1}y_{1}dS(y)\\
  	& \quad + 2\displaystyle\int_{\partial B_{r}\cap\{|y| < r/4\}}k_{2} \dfrac{(h-y_{2})^{2} - y_{1}^{2}}{|x-y|^{4}}a_{1}y_{1}dS(y)\\
  	& \quad + \frac{C}{r}\displaystyle\int_{\partial B_{r}\cap\{|y| < r/4\}} \dfrac{((h-y_{2})^{2} - y_{1}^{2})y_{1}^{2}}{((h-y_{2})^{2} + y_{1}^{2})^{2}}dS(y)\\
  	& \leq C + 0+ \frac{C}{r}\displaystyle\int_{\partial B_{r}\cap\{|y| < r/4\}}dS(y) \leq C.
  	\end{align*}
  	Note that $ \displaystyle\int_{\partial B_{r}\cap\{|y| < r/4\}}k_{2} \dfrac{(h-y_{2})^{2} - y_{1}^{2}}{|x-y|^{4}}a_{1}y_{1}dS(y) = 0$, because the integrand is odd in $y_{1}$.\\
  	
  	Similarly, we can derive that $\partial_{12}K_{21}(x) \leq C$. We leave details to interested readers.
  	
  \end{proof}

   Now we are ready to prove Theorem \ref{singlepatch}.\\

	\begin{proof}[Proof of Theorem \ref{singlepatch}]
	
	Let $A(t) := \dfrac{A_{\gamma}(t) + A_{\infty}(t)}{A_{\inf}(t)}$. From inequalities (\ref{Ainftyestimate}), (\ref{Ainfestimate}), (\ref{Agammaestimate}), \eqref{graduest}, and \eqref{secondkey531},
we have
	\[
	A^{\prime}(t) \leq C(D,\gamma)A(t)(1 + \log_{+}A(t)).
	\]
	We thus obtain that $A(t)$ grows at most double-exponentially in time, and therefore, the same estimate applied to $\dfrac{A_{\gamma}(t)}{A_{\inf}(t)}$.
 Given that, double exponential upper bound on growth can be obtained for $A_{\infty}(t)$, $A_{\inf}(t)^{-1}$ and $A_{\gamma}(t)$ from
(\ref{Ainftyestimate}), (\ref{Ainfestimate}) and (\ref{Agammaestimate}) respectively. The proof is completed.
    \end{proof}

    \section{General case}

    In this section, we consider the general case, where the initial data is
    \[
    \omega_{0}(x) = \sum\limits_{k=1}^{N}\theta_{k}\chi_{\Omega_{k}(0)}(x).
    \]

    By Yudovich theory (see \cite{Yudovich}, \cite{Majda} or \cite{Marchioro}), there exists a unique solution in the form of
    \begin{equation}\label{patchsolution}
    \omega(x,t) := \sum\limits_{k=1}^{N}\theta_{k}\chi_{\Omega_{k}(t)}(x),
    \end{equation}
    with $\Omega_{k}(t) = \Phi_{t}(\Omega_{k}(0))$ for each $k$. Also note that $\Phi_{t}(x)$ is uniquely defined for any $x \in \mathbb{R}^{2}$, due to time independent log-Lipschitz bound
    \begin{equation}\label{loglip}
    |u(x,t) - u(y,t)| \leq C(D)\norm{\omega_{0}}_{L^{\infty}}|x-y|\log(1 + |x-y|^{-1}).
    \end{equation}

    By Definition \ref{generalpatchdef}, to show that $\omega$ in (\ref{patchsolution}) is a $C^{1,\gamma}$ patch solution, we need to prove that
$\{\partial\Omega_{k}(t)\}_{k=1}^{N}$ is a family of disjoint simple closed curves for each $t \geq 0$, and
    \[
    \sup\limits_{t \in [0,T]}\max\limits_{k}\norm{\partial \Omega_{k}(t)}_{C^{1,\gamma}} \, < \infty
    \]
    for each $T < \infty$.\\

    First note that (\ref{loglip}) yields
    \[
    \min\limits_{i\not = k}\text{dist}(\Omega_{i}(t), \Omega_{k}(t)) \geq \delta(t) > 0
    \]
    for all $t \geq 0$, where $\delta(t)$ decreases at most double exponentially in time. This is going to ensure that the effects of the patches on each other will be controlled. Now, it remains to prove that each $\partial\Omega_{k}(t)$ is a simple closed curve with $\parallel\partial\Omega_{k}(t)\parallel_{C^{1,\gamma}}$ uniformly bounded on bounded time interval.\\

 Next, we add $\sup\limits_{k}$ in the definitions of $A_{\infty}$, $A_{\gamma}$ and add $\inf\limits_{k}$ in the definition of $A_{\inf}$.
    Let us decompose
    \[
    u = \sum\limits_{i=1}^{N}u_{i},
    \]
    with each $u_{i}$ coming from the contribution of the patch $\Omega_{i}$ to $u$. If $i \not = k$, then we have
    \[
    \norm{\nabla^{n}u_{i}(\cdot,t)}_{L^{\infty}(\Omega_{k}(t))} \, \leq C(\omega_{0},n,D)\delta(t)^{-n-1}
    \]
    for all $n \geq 0$. This yields
    \[
    \norm{\nabla u_{i}(\cdot,t)}_{\dot{C}^{\gamma}(\Omega_{k}(t))} \, \leq C(\omega_{0},D)\delta(t)^{-3}.
    \]

    Analogously to Proposition \ref{gradientofvelocity}, we also have the estimate by simple scaling,
    \begin{equation}\label{generalgradientv}
    \norm{\nabla u_{i}(\cdot,t)}_{L^{\infty}(\mathbb{R}^{2})} \leq C(D,\gamma)|\theta_{i}|\Big(1+\log_{+}\frac{A_{\gamma}(t)}{A_{\inf}(t)}\Big).
    \end{equation}

    With all these in hand, let us prove Theorem \ref{generalpatch}.
    \begin{proof}[Proof of Theorem \ref{generalpatch}]
    	We now consider $\varphi_{k}$ and $w_{k} := \nabla^{\perp}\varphi_{k}$ for each $\Omega_{k}$. With $\Theta := \max\limits_{1 \leq k \leq N}|\theta_{k}|$, for each $k$ and $t > 0$, we have
    	
    	\begin{tabbing}
    		$\qquad$ $\norm{(\nabla u)w_{k}}_{\dot{C}^{\gamma}(\Omega_{k})}$ \= $\leq$ $C(D,\gamma)\Theta A_{\gamma}\Big(1 + \log_{+}\dfrac{A_{\gamma}}{A_{\inf}}\Big) + A_{\infty}$\\
    		\> $\quad$ $+$ $\sum\limits_{i \not = k}\norm{\nabla u_{i}}_{L^{\infty}(\Omega_{k})}\norm{w_{k}}_{\dot{C}^{\gamma}(\Omega_{k})}$\\
    		\> $\quad$ $+$ $\sum\limits_{i \not = k}\norm{\nabla u_{i}}_{\dot{C}^{\gamma}(\Omega_{k})}\norm{w_{k}}_{L^{\infty}(\Omega_{k})}$\\
    		\> $\leq$ $C(D,\gamma)N\Theta A_{\gamma}\Big(1 + \log_{+}\dfrac{A_{\gamma}}{A_{\inf}}\Big) + C(D,\omega_{0})N\delta(t)^{-3}A_{\infty}.$
    		
    	\end{tabbing}

    	Then we have estimates,
    	\[
    	A_{\gamma}^{\prime}(t) \leq C(D,\gamma)N\Theta A_{\gamma}(t)\Big(1 + \log_{+}\frac{A_{\gamma}(t)}{A_{\inf}(t)}\Big) + C(\omega_{0},D)N\delta(t)^{-3}A_{\infty}(t).
    	\]
    	Let $\widetilde{A}(t) := A_{\gamma}(t)A_{\inf}(t)^{-1} + A_{\infty}(t)$, then a simple computation yields that
    	\[
    	\widetilde{A}^{\prime}(t) \leq C(D,\gamma,\omega_{0})\widetilde{A}(t)\big(\delta(t)^{-3} + \log_{+}\widetilde{A}(t)\big).
    	\]
    	
    	Since $\delta(t)^{-3}$ increases at most double exponentially in time, it follows that $\widetilde{A}(t)$ increases at most triple exponentially. So $\norm{\partial\Omega_{k}(t)}_{C^{1,\gamma}}$ is uniformly bounded on bounded  time intervals, thus completing the proof.\\
    \end{proof}

	\section{One special case with double exponential upper bound}
	We consider a special case in a unit disc $D := B_{1}(0)$ with initial data in the following form:
	\[
	\omega_{0}(x) = \omega_{1}(x,0) - \omega_{2}(x,0) = \chi_{\Omega_{1}(0)}(x) - \chi_{\Omega_{2}(0)}(x).
	\]
	Here $\Omega_{1}(0)$ and $\Omega_{2}(0)$ are two single disjoint patches that are symmetric with respect to the line $x_{1} = 0$. 
The Euler evolution preserves the odd symmetry, so the solution is of the form
	\[
	\omega(x,t) = \omega_{1}(x,t) - \omega_{2}(x,t) = \chi_{\Omega_{1}(t)}(x) - \chi_{\Omega_{2}(t)}(x)
	\]
	for all times, where $\Omega_{1}(t)$ and $\Omega_{2}(t)$ are two symmetric single disjoint patches.\\
	
	Note that when $D$ is a disk, we have an explicit formula for $G_{D}(x,y)$. The velocity $u$ generated by single patch $\Omega$ is given by
	
	\begin{equation}
		u = v + \widetilde{v} = -\dfrac{1}{2\pi}\int_{\Omega}\dfrac{(x-y)^{\perp}}{|x-y|^{2}}dy + \dfrac{1}{2\pi}\int_{\widetilde{\Omega}}\dfrac{(x-y)^{\perp}}{|x-y|^{2}}\dfrac{1}{|y|^{4}}dy.
	\end{equation}
	Note that the reflection map is defined via $\widetilde{y} = S(y) = \dfrac{y}{|y|^{2}}$ and the Jacobian $F(y) = \dfrac{1}{|y|^{4}}$. Here $S(y)$ is defined on all $D$, not only restricted to $T(r)$. Also notice that $T(r)$ is the same as the annulus $A(0;1-r, 1+r):= \{x: 1-r \leq |x| \leq 1+r\}$. We will keep using $T(r)$ instead of the annulus $A$ for convenience and consistency.
In the argument below, $r$ will a sufficiently small universal constant.

	Let us prove Theorem \ref{symmetriccase}.

	\begin{proof}[Proof of Theorem \ref{symmetriccase}]
	
	From now on, we will drop $t$ from $\Omega_{k}(t)$, since the estimate is time independent. We adopt $\varphi_{k}$ and $w_{k}$ notation for $k = 1,2$ from Section 3 and we also add $\sup\limits_{k}$ in the definitions of $A_{\infty}$, $A_{\gamma}$ and add $\inf\limits_{k}$ in the definition of $A_{\inf}$ as we did in Section 3. Note that $u = u_{1} + u_{2}$, and for $i = 1, 2$,
	\[
	u_{i} = v_{i} + \widetilde{v}_{i} = -\dfrac{1}{2\pi}\int_{\Omega_{i}}\dfrac{(x-y)^{\perp}}{|x-y|^{2}}dy + \dfrac{1}{2\pi}\int_{\widetilde{\Omega}_{i}}\dfrac{(x-y)^{\perp}}{|x-y|^{2}}\dfrac{1}{|y|^{4}}dy.
	\]
	
	From the proof in Section 3, our goal is to estimate $\norm{(\nabla u)w_{k}}_{\dot{C}^{\gamma}(\Omega_{k} \cap T(r/2))}$ for $k = 1, 2$. Without loss of generality, it suffices to only estimate $\norm{(\nabla u)w_{1}}_{\dot{C}^{\gamma}(\Omega_{1} \cap T(r/2))}$. We decompose it to be sum of $\norm{(\nabla u_{1})w_{1}}_{\dot{C}^{\gamma}(\Omega_{1} \cap T(r/2))}$ and $\norm{(\nabla u_{2})w_{1}}_{\dot{C}^{\gamma}(\Omega_{1} \cap T(r/2))}$, and estimate them one by one.
	
	For the first term $\norm{(\nabla u_{1})w_{1}}_{\dot{C}^{\gamma}(\Omega_{1} \cap T(r/2))}$, $u_{1}$ and $w_{1}$ are both generated by the patch $\Omega_{1}$, and the argument is identical to the single patch case in Section 2. Thus we have,
	\[
	\norm{(\nabla u_{1})w_{1}}_{\dot{C}^{\gamma}(\Omega_{1} \cap T(r/2))} \leq C(\gamma)(A_{\gamma} + A_{\infty})\Big(1+\log_{+}\frac{A_{\gamma}}{A_{\inf}} \Big).
	\]
	
	For the second term $\norm{(\nabla u_{2})w_{1}}_{\dot{C}^{\gamma}(\Omega_{1} \cap T(r/2))}$, we decompose $(\nabla u_{2})w_{1}$ as
	\[
	(\nabla u_{2})w_{1} = (\nabla v_{2})w_{1} + (\nabla\widetilde{v}_{2})w_{1}.
	\]
	
	First we claim that $\norm{(\nabla v_{2})w_{1}}_{\dot{C}^{\gamma}(\Omega_{1} \cap T(r/2))}$ can be bounded by $C(\gamma)A_{\gamma}\Big(1+\log_{+}\dfrac{A_{\gamma}}{A_{\inf}}\Big)$. Indeed, $v_{2}$ can be regarded as the velocity field generated by the patch $\Omega_{2}$ in $\mathbb{R}^{2}$. $w_{1}$ is a divergence free vector field that is tangent to the boundary of $\Omega_{1}$, which is symmetric to $\Omega_{2}$ over $x_{1} = 0$. This case has been treated in \cite{KRYZ} from page 15 to the end of the Section 3. Note that the symmetry here is with respect to $x_{1} = 0$, while in \cite{KRYZ} the symmetry is with respect to $x_{2} = 0$.

	For the second term, denote $g(x) := (\nabla\widetilde{v}_{2})w_{1}(x)$. Then for arbitrary $x, x^{\prime} \in \Omega_{1}\cap T(r/2)$, we have
	
	\begin{equation}\label{Cgammaofg}
	\begin{split}
	\dfrac{|g(x)-g(x^{\prime})|}{|x-x^{\prime}|^{\gamma}} & \leq |\nabla\widetilde{v}_{2}(x^{\prime})|\norm{w_{1}}_{\dot{C}^{\gamma}(\Omega_{1})}\\
	& \quad + \dfrac{|\nabla\widetilde{v}_{2}(x) - \nabla\widetilde{v}_{2}(x^{\prime})|}{|x-x^{\prime}|^{\gamma}}|w_{1}(x)|.
	\end{split}
	\end{equation}

	The first term in (\ref{Cgammaofg}) can be easily bounded by $C(\gamma)A_{\gamma}\Big(1+\log_{+}\dfrac{A_{\gamma}}{A_{\inf}}\Big)$, by Proposition \ref{gradientofvelocity} and definition of $A_{\gamma}$ and $w_{1}$.
	
	For the second term, remember that $\widetilde{v}_{2} = \dfrac{1}{2\pi}\displaystyle\int_{\widetilde{\Omega}_{2}}\dfrac{(x-y)^{\perp}}{|x-y|^{2}}\dfrac{1}{|y|^{4}}dy$, where $\widetilde{\varphi}_{2}$ defines $\widetilde{\Omega}_{2}$ and $\widetilde{w}_{2} = \nabla^{\perp}\widetilde{\varphi}_{2}$. Then by Proposition \ref{T2estimate}, we have
	
	\begin{equation}\label{symmetriccaseinequality}
	\dfrac{|\nabla\widetilde{v}_{2}(x) - \nabla\widetilde{v}_{2}(x^{\prime})|}{|x-x^{\prime}|^{\gamma}} \leq C(\gamma)\Big(1 + \log_{+}\dfrac{A_{\gamma}}{A_{\inf}}\Big)\min\Big\{\dfrac{A_{\gamma}}{|\widetilde{w}_{2}(P_{x})|}, d(x)^{-\gamma}\Big\}.
	\end{equation}
	
	Here $d(x) = \text{dist}(x, \widetilde{\Omega}_{2})$ and $P_{x} \in \partial \widetilde{\Omega}_2$ denotes the closest point. Note that $\widetilde{v}_{2}$ plays the same role as $I_{12}$ in Proposition \ref{T2estimate}. There is a minor difference that the region of integral for $\widetilde{v}_{2}$ is not restricted to $T(r)$, since we have a more explict expression for Green's function in unit disk. However, the same arguments work here.

	\begin{figure}[h]
	\includegraphics[scale=0.5]{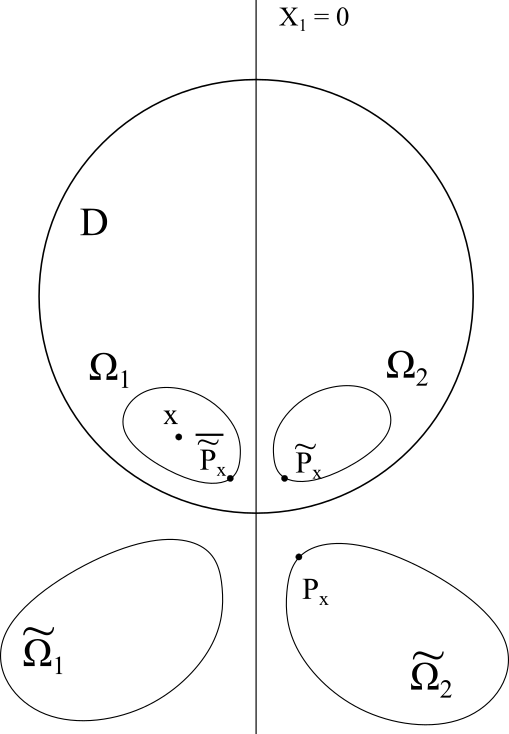}
	\caption{}
	\label{fig6}
	\end{figure}

	The symmetry of $\Omega_{1}$ and $\Omega_{2}$ implies that we can choose $\phi_{1,2}$ so that 
	\[
	\varphi_{1}(x,t) = \varphi_{2}(\bar{x},t),
	\]
for all $t,$	where $\bar{x} = (-x_{1}, x_{2})$. This is the key observation that allows to reduce the upper bound from triple to double exponential growth. Therefore, by definition $w_{i} = \nabla^{\perp}\varphi_{i}$, we have
	\[
	w_{2}(x) = -\overline{w_{1}(\bar{x})}.
	\]
	For any $x \in \Omega_{1}$ (see Figure \ref{fig6})
	\[
	|w_{1}(x)| \leq |w_{1}(x) - w_{1}(\overline{\widetilde{P}_{x}})| + |w_{1}(\overline{\widetilde{P}_{x}})|.
	\]
	From above, we know $|w_{1}(\overline{\widetilde{P}_{x}})| = |w_{2}(\widetilde{P}_{x})|$. Without loss of generality, 
we can choose $r$ small enough so that if $x \in T(r/2)$, then we have $|x - \overline{\widetilde{P}_{x}}| \leq Cd(x)$. Thus we obtain,
	\[
	|w_{1}(x)| \leq CA_{\gamma}d^{\gamma} + C|\widetilde{w}_{2}(P_{x})|.
	\]
	Therefore,
	\[
	\norm{(\nabla\widetilde{v}_{2})w_{1}}_{\dot{C}^{\gamma}(\Omega_{1} \cap T(r/2))} \leq C(\gamma)(A_{\gamma} + A_{\infty})\Big(1+\log_{+}\dfrac{A_{\gamma}}{A_{\inf}}\Big).
	\]
	Thus, we get
	\[
	\norm{(\nabla u)w}_{\dot{C}^{\gamma}(\Omega_{1} \cap T(r/2))} \leq C(\gamma)(A_{\gamma}+ A_{\infty})\Big(1+\log_{+}\dfrac{A_{\gamma}}{A_{\inf}}\Big).
	\]
	We conclude that $A_{\gamma}$, $A_{\inf}^{-1}$ and $A_{\infty}$ grow at most double exponentially in time.
	
	\end{proof}

	\section{Example with double exponential growth}

	In this section, we are going to use an analog of the example constructed in \cite{KS} to show that the upper bound obtained by the previous section is actually 
qualitatively sharp.\\
	
	First we introduce some notation that will be adopted throughout this section. With $\phi$ to be the usual angular variable, we have
	\begin{align*}
	D &:= B_{1}(e_{2}), \quad \textrm{with} \quad e_{2} = (0,1),\\
	D^{+} & := \{ (x_{1},x_{2}) \in D : x_{1} \geq 0\},\\
	D_{1}^{\gamma} & :=  \{(x_{1}, x_{x}) \in D^{+}|\frac{\pi}{2}-\gamma \geq \phi \geq 0\},\\
	D_{2}^{\gamma} & :=  \{(x_{1}, x_{2}) \in D^{+}|\frac{\pi}{2} \geq \phi \geq \gamma\},\\
	Q(x_{1},x_{2}) & :=  \{(y_{1}, y_{2})\in D^{+}|y_{1} \geq x_{1}, y_{2} \geq x_{2}\},\\
	\Omega(x_{1},x_{2},t) & :=  \dfrac{4}{\pi}\displaystyle\int_{Q(x_{1},x_{2})}\frac{y_{1}y_{2}}{|y|^{4}}\omega(y,t)dy.
	\end{align*}

	Consider two-dimensional Euler equation on $D$, let $\omega$ be vorticity. We will take smooth patch initial data $\omega_{0}$ ao that $\omega_{0}(x) \geq 0$ for $x_{1} > 0$ and $\omega_{0}$ is odd in $x_{1}$. Let us state the Key Lemma (see \cite{KS} for Lemma 3.1).

	\begin{lemma}\label{keylemma}
	Take any $\gamma$, $\pi/2 > \gamma > 0$. Then there exists $\delta > 0$ such that
	\begin{equation}\label{u1est531}
	u_{1}(x) = -x_{1}\Omega(x_{1},x_{2}) + x_{1}B_{1}(x), \hspace{1mm}|B_{1}| \leq C(\gamma)\norm{\omega_{0}}_{L^{\infty}}, \hspace{2mm} \forall x \in D_{1}^{\gamma}, |x| \leq \delta
	\end{equation}
	\begin{equation}\label{u2est531}
	u_{2}(x) = x_{2}\Omega(x_{1},x_{2}) + x_{2}B_{2}(x), \hspace{1mm}|B_{2}| \leq C(\gamma)\norm{\omega_{0}}_{L^{\infty}}, \hspace{2mm} \forall x \in D_{2}^{\gamma}, |x| \leq \delta
	\end{equation}
	\end{lemma}

	Note that, in \cite{KS} Lemma \ref{keylemma} applies only to smooth $\omega$, but at the same time the argument can extend to patches without any effort. Following the proof in \cite{KS}, exponential growth of curvature can be achieved easily. Indeed, take initial data $\omega_{0}(x)$ which is equal to $1$ everywhere in $D^{+}$ except on a thin strip of width equal to $\delta/2$ ($\delta$ is chosen from Lemma \ref{keylemma}, for some small $\gamma < \pi/10$) near the vertical axis $x_{1} = 0$, where $\omega_{0}(x) = 0$. Then we round the corner of this single patch to make the boundary so that $\omega_{0}$ remains equal to one everywhere in $D^{+}$ except on a thin strip of width equal to at most $\delta$ near the vertical axis $x_{1} = 0$. Denote the patch in $D^{+}$ at the initial time by $P(0)$, so $\omega_{0}(x) = \chi_{P(0)}(x) - \chi_{\bar{P}(x)}(x)$, where $\bar{P} := \big\{(x_{1}, x_{2}): (-x_{1},x_{2}) \in P\big\} $. By odd symmetry, two single patches $P$ and $\bar{P}$ will stay in the two half disks respectively for all time $t$. Due to incompressibility, the measure of the set in $D^+$ where $\omega(x,t) = 0$ does not exceed $2\delta$. In this case, for every $x \in D^{+}$ with $|x| < \delta$, we can derive the following estimate for $\Omega(x_{1}, x_{2})$,
	\[
	\Omega(x_{1},x_{2},t) \geq \int_{2\delta}^{2}\int_{\pi/6}^{\pi/3}\omega(r,\phi)\frac{\sin 2\phi}{2r}d\phi dr \geq \frac{\sqrt{3}}{4}\int_{2\delta}^{2}\int_{\pi/6}^{\pi/3}\frac{\omega(r,\phi)}{r}.
	\]
	The value of the integral on the right hand side is minimal when the area where $\omega(r,\phi) = 0$ is situated around small values of the radial variable. Since this area does not exceed $2\delta$, we have
	\begin{equation}\label{logdelta}
	\frac{4}{\pi}\Omega(x_{1},x_{2},t) \geq c_{1}\int_{c_{2}\sqrt{\delta}}^{1}\int_{\pi/6}^{\pi/3}\frac{1}{r}d\phi dr \geq C_{1}\log \delta^{-1},
	\end{equation}
	where $c_{1}$, $c_{2}$ and $C_{1}$ are positive universal constants.\\

Let $x_0 \in \partial D,$ $0 < x_{0,1} < \delta$ be the point with minimal value of $x_{0,1}$ such that $\omega_0(x_0)=1.$ Consider the trajectory $\Phi_t(x_0);$ due to the boundary conditions 
$\Phi_t(x) \in \partial D$ for all times. If $\delta$ is sufficiently small (note that we can always make it smaller if needed), the estimates \eqref{u1est531} and \eqref{logdelta} 
imply that the first component of $\Phi_t(x_0),$ $\Phi^1_t(x_0),$ converges to zero at an exponential rate. 
Observe  that for a curve, one can use the distance over which it changes direction by $\pi/2$ to estimate the curvature. In our case, $\Phi_{t}(x_{0})$ is being pushed towards the origin from the along the boundary $\partial D$, and by odd symmetry, the axis $x_{1} = 0$ is a barrier that patch $P$ can not pass. So the tangent vector to $\partial P$ has to turn close to $\pi/2$ angle 
over a distance that is exponentially decaying in time. 

	To achieve double exponential growth on curvature, by discussion above, we need an example where we can track a point on the patch boundary that approaches the origin at double exponential speed.\\
	
	Let us prove Theorem \ref{sharpexample} by using the example constructed in \cite{KS}; the proof is similar to \cite{KS}.


	\begin{proof}[Proof of Theorem \ref{sharpexample}]
	We first fix some small $\gamma > 0$ (take $\pi/10$ for example).  We choose $\delta > 0$ small enough such that Lemma \ref{keylemma} applies and that $C_{1}\log\delta^{-1} > 100C(\gamma)$ with $C_{1}$ from (\ref{logdelta}) and $C(\gamma)$ from Lemma \ref{keylemma}. We take the smooth patch initial data $\omega_{0}$ as the one constructed in the previous exponential growth example, with $\omega_{0} = 1$ everywhere in $D^{+}$ except on a thin strip near $x_1=0$ line, except now with width equal to $\delta^{10}$.\\
	
	For $0 < x^{\prime}_{1}, x^{\prime\prime}_{1} < 1$, we denote
	\[
	\mathcal{O}(x^{\prime}_{1},x^{\prime\prime}_{1}) = \big\{(x_{1}, x_{2})\in D^{+}|x^{\prime}_{1} \leq x \leq x^{\prime\prime}_{1}, x_{2} < x_{1}\big\}.
	\]
	For $0 < x_{1} < 1$, we let
	\begin{equation}\label{fastslowspeed}
	\begin{split}
	\overline{u}(x_{1},t) &= \max\limits_{(x_{1},x_{2}) \in D^{+}x_{2} < x_{1}}u_{1}(x_{1},x_{2},t), \\
	\underline{u}(x_{1},t) &= \min\limits_{(x_{1},x_{2}) \in D^{+}x_{2} < x_{1}}u_{1}(x_{1},x_{2},t),
	\end{split}
	\end{equation}

	and define $a(t)$, $b(t)$ by
	\begin{align*}
	a^{\prime}(t) &= \overline{u}(a(t),t),  \quad a(0) = \delta^{10},\\
	b^{\prime}(t) &= \underline{u}(b(t),t), \quad b(0) = \delta.	
	\end{align*}
	
This set up is a little simpler than in \cite{KS}, where an additional scale $\epsilon$ was introduced. However, the rest of the argument remains the same. 
The estimates \eqref{u1est531}, \eqref{u2est531} can be used to control the region $\mathcal{O}_{a(t),b(t)},$ to show that it does not collapse, 
and to trace its approach to the origin. The result is an improvement of the bound \eqref{logdelta} bound yielding exponential growth of $\Omega(a(t),0).$
This leads to double exponential decay of $a(t).$ Since the argument is largely parallel to \cite{KS}, we refer to it for the details.   
		
	\end{proof}

	\section{Appendix}
Here we give a sketch of proof for Proposition \ref{estimatesfromxiao}. 
\begin{proof}[Proof of Proposition \ref{estimatesfromxiao}]
Recall the representation 
\[ G_D(x,y) = \frac{1}{2\pi} \left( \log |x-y| - \log |x - \widetilde{y}| \right) + B(x,y) \]
for $y \in T(r).$ Note that $B(x,y)$ satisfies 
\[ \Delta_x B(x,y) =0, \,\,\,\, \left. B(x,y) \right|_{x \in \partial D} = \frac{1}{2\pi} \log \frac{|x-\widetilde{y}|}{|x-y|}. \]
In \cite{Xu}, the following estimate is proved. Let $D \in C^3,$ and fix any $z \in \partial D.$ Then there exists $r = r(D)$ such that for any $\omega \in L^{\infty}(\bar D)$ 
\begin{eqnarray} \nonumber \norm{\int_{B_r(z) \cap D} B(x,y) \omega(y)\,dy}_{C^{2,\alpha}(\partial D \cap B_{r/2}(z))} = \\ \label{xuest}
\norm{ \frac{1}{2\pi} \int_{B_r(z) \cap D} \log \frac{|x-\widetilde{y}|}{|x-y|} \omega(y)\,dy }_{C^{2,\alpha}(\partial D \cap B_{r/2}(z))} \leq C\|\omega\|_{L^\infty(\bar D)}. \end{eqnarray}
This estimate is a consequence of a calculation in the proof of Proposition 1 and Lemma 4 in \cite{Xu}. The argument is fairly direct and uses
estimates on $\widetilde{y}(y)$ and local representation of $\partial D$ as a graph of a function $f \in C^3;$ the idea is that when $x \in \partial D \cap B_{r/2}(z)$ and 
$y \in D \cap B_r(z)$ then $|x-y|$ and $|x-\widetilde{y}|$ are very close. 
Note that even though the statement 
of Proposition 1 in \cite{Xu} makes an assumption that $D$ has a symmetry axis and $z$ belongs to it, this assumption is never used in the proof 
(it is needed for later applications in \cite{Xu}). 

Now it is straightforward to extend the estimate \eqref{xuest} to the function 
\begin{equation}\label{varphidef} \varphi(x) := \frac{1}{2\pi}\int_{T(r) \cap D} \log \frac{|x-\widetilde{y}|}{|x-y|} \omega(y)\,dy, \end{equation}
yielding 
\begin{equation}\label{varphiimp} \| \varphi \|_{C^{2,\alpha}(\partial D \cap B_{r/2}(z))} \leq C\|\omega\|_{L^\infty}. \end{equation}
Indeed, in \eqref{varphidef} compared with \eqref{xuest} we are adding integration over a region where $y$ satisfies $|y-x| \geq r/2,$ and if we choose $r=r(D)$ 
sufficiently small, also $|\widetilde{y}-x| \geq r/4.$ Then $\log \frac{|x-\widetilde{y}|}{|x-y|}$ is a smooth function of $x$ with uniform bounds on derivatives for all such 
$y \in T(r),$ leading to \eqref{varphiimp}. In fact, since $z$ was arbitrary, we get that 
\[ \|  \varphi \|_{C^{2,\alpha}(\partial D)} \leq C\|\omega\|_{L^\infty(\bar D)}. \] 
By well known results (e.g. Lemma 6.38 of \cite{Trudinger}, there exists an extension $\varphi^e$ of $\varphi$ to $\bar D$ such that 
\[  \|  \varphi^e \|_{C^{2,\alpha}(\bar D)} \leq \|  \varphi \|_{C^{2,\alpha}(\partial D)}. \] 
Now observe that 
\[ g(x) :=  \int_{T(r) \cap D} B(x,y) \omega(y)\,dy     \] 
defined for all $x \in \bar D$ satisfies \[ \Delta g =0, \,\,\, \left. g \right|_{\partial D} = \varphi^e. \]
By Theorems 6.8 and 6.6 of \cite{Trudinger} we obtain 
\[ \|g\|_{C^{2,\alpha}(\bar D)} \leq C \|\varphi \|_{C^{2,\alpha}(\partial D)} \leq C \|\omega\|_{L^\infty(\bar D)}, \]
completing the proof.
\end{proof}

    \section*{Acknowledgement}
	We would like to thank Yao Yao for valuable advice. Partial support of the NSF-DMS grants 1412023 and 1712294 is gratefully acknowledged.

\end{document}